\documentclass[11pt,letterpaper]{scrarticle}
\usepackage[utf8]{inputenc}
\usepackage[T1]{fontenc}
\usepackage{amsmath}
\usepackage{amsfonts}
\usepackage{amssymb}
\usepackage{amsthm}
\usepackage{mathtools}
\usepackage{euscript}
\usepackage[dvipsnames]{xcolor}

\usepackage{hyperref}
\usepackage{subcaption}
\usepackage[inner=1in,outer=1in]{geometry}
\usepackage{tikz}
\usepackage{tikz-cd}
\usetikzlibrary{positioning, calc}
\usetikzlibrary{decorations.markings}
\tikzstyle{vertex}=[inner sep=0pt]

\usepackage[backend=biber]{biblatex}
\theoremstyle{plain}
\newtheorem{thm}{Theorem}[section]
\newtheorem*{thm*}{Theorem}
\newtheorem{prop}[thm]{Proposition}
\newtheorem{lem}[thm]{Lemma}

\theoremstyle{definition}
\newtheorem{defn}[thm]{Definition}
\newtheorem{rmk}[thm]{Remark}

\theoremstyle{remark}
\newtheorem{example}[thm]{Example}

\numberwithin{equation}{section}

\def\on{\operatorname}
\def\op{{\on{op}}}
\def\Hom{{\on{Hom}}}

\def\Span{{\on{Span}}}
\def\Rel{{\on{Rel}}}
\def\Cat{{\on{Cat}}}
\def\Set{{\on{Set}}}
\def\Vect{{\on{Vect}}}
\def\Cut{{\on{Cut}}}

\def\scr{\EuScript}
\def\C{{\scr{C}}}

\def\Fin{{\scr{F}\!\on{in}}}
\newcommand{\fin}[1]{{\langle #1\rangle}}

\newcommand{\id}{\mathrm{id}}
\renewcommand{\emptyset}{\varnothing}
\newcommand{\suchthat}{\colon}
\newcommand{\bitimes}[2]{\,\vphantom{\times}_{#1}\!\! \times_{#2}\!}
\newcommand{\pt}{\{\bullet\}}
\DeclareMathOperator{\Aut}{Aut}
\renewcommand{\hat}[1]{\widehat{#1}}
\newcommand{\out}{\mathrm{out}}

\makeatletter
\tikzcdset{
	open/.code     = {\tikzcdset{hook, circled};},
	open'/.code    = {\tikzcdset{hook', circled};},
	circled/.code  = {\tikzcdset{markwith = {\draw (0,0) circle (.375ex);}};},
	markwith/.code ={
		\pgfutil@ifundefined%
		{tikz@library@decorations.markings@loaded}%
		{\pgfutil@packageerror{tikz-cd}{You need to say %
				\string\usetikzlibrary{decorations.markings} to use arrows with markings}{}}{}%
		\pgfkeysalso{/tikz/postaction = {
				/tikz/decorate,
				/tikz/decoration={markings, mark = at position 0.5 with {#1}}}
		}
	},
}
\makeatother

\newcommand{\stringdiagram}[1]{\[\begin{tikzpicture}[scale=.33, thick]
#1\end{tikzpicture}\]}
\newcommand{\stringdiagramlabel}[2]{\begin{equation}\label{#2}\begin{tikzpicture}[scale=.33, thick, baseline=(current  bounding  box.center)]
#1\end{tikzpicture}\end{equation}}
\newcommand{\identity}[2]{\draw (#1,#2+1) -- (#1,#2-1);}
\newcommand{\unit}[2]{\draw (#1,#2) circle [radius=.1]; \draw (#1,#2-.1) -- (#1,#2-1);}
\newcommand{\counit}[2]{\draw (#1,#2) circle [radius=.1]; \draw (#1,#2+.1) -- (#1,#2+1);}
\newcommand{\morphism}[3]{\draw (#1-.4, #2-.6) rectangle (#1+.4,#2+.6); \draw (#1,#2-.6) -- (#1,#2-1); \draw (#1,#2+.6) -- (#1, #2+1); \node at (#1,#2) {$#3$};}
\newcommand{\multiplication}[2] {\draw (#1,#2-.3) -- (#1-1,#2+1); \draw (#1,#2-.3) -- (#1+1,#2+1); \draw (#1,#2-.3) -- (#1,#2-1);}
\newcommand{\slimmultiplication}[2] {\draw (#1,#2-.3) -- (#1-.5,#2+1); \draw (#1,#2-.3) -- (#1+.5,#2+1); \draw (#1,#2-.3) -- (#1,#2-1);}
\newcommand{\tripleprod}[2] {\draw (#1,#2) -- (#1-1,#2+1); \draw (#1,#2) -- (#1+1,#2+1); \draw (#1,#2+1) -- (#1,#2-1);}

\newcommand{\swaptwoc}[2]{\draw (#1-1,#2+1) -- (#1+1,#2-1); \draw (#1-1,#2-1) -- (#1,#2+1); \draw (#1,#2-1) -- (#1+1,#2+1); }
\newcommand{\copairing}[2]{\draw (#1+1,#2-1) arc(0:180:1);}
\newcommand{\pairing}[2]{\draw (#1-1,#2+1) arc(180:360:1);}
\newcommand{\swap}[2]{\draw (#1-1,#2+1) -- (#1+1,#2-1); \draw (#1-1,#2-1) -- (#1+1,#2+1);}
\newcommand{\slimswap}[2]{\draw (#1-.5,#2+1) -- (#1+.5,#2-1); \draw (#1-.5,#2-1) -- (#1+.5,#2+1);}
\newcommand{\swaptwoa}[2]{\draw (#1-1.5,#2+1) -- (#1+1.5,#2-1); \draw (#1-1.5,#2-1) -- (#1,#2+1); \draw (#1+.5,#2-1) -- (#1+1.5,#2+1); }
\newcommand{\swaptwob}[2]{\draw (#1-.5,#2+1) -- (#1+2.5,#2-1); \draw (#1-.5,#2-1) -- (#1+1,#2+1); \draw (#1+.5,#2-1) -- (#1+2,#2+1); }
\newcommand{\equals}[2]{\node at (#1,#2) {$=$};}
\newcommand{\nattrans}[3]{\node at (#1,#2) {$\xRightarrow{#3}$};}
\newcommand{\leftnattrans}[3]{\node at (#1,#2) {$\xLeftarrow{#3}$};}
\newcommand{\upnattrans}[3]{\node at (#1,#2) {$\big\Uparrow #3$};}
\newcommand{\highlight}[4]{\draw[teal, dotted] (#1,#2) rectangle (#3,#4);}
\newcommand{\enclose}[4]{\draw[brown, thin] (#1,#2) rectangle (#3,#4);}

\usepackage{rotating}

\title{Frobenius and commutative pseudomonoids in the bicategory of spans}
\author{Ivan Contreras\thanks{Department of Mathematics, Amherst College, 31 Quadrangle Drive, Amherst,
		MA 01002 \\ Email: icontreraspalacios@amherst.edu} \and Rajan Amit Mehta\thanks{Department of Mathematical Sciences, Smith College, 44 College Lane,
		Northampton, MA 01063 \\ Email: rmehta@smith.edu} \and Walker H. Stern\thanks{Department of Mathematics, Bilkent University, Ankara, Turkey \\ Email: walker@walkerstern.com}}

\begin{document}
	\maketitle 
	
	\begin{abstract}
 \noindent
	In previous work by the first two authors, Frobenius and commutative algebra objects in the category of spans of sets were characterized in terms of simplicial sets satisfying certain properties. In this paper, we find a similar characterization for the analogous coherent structures in the bicategory of spans of sets. We show that commutative and Frobenius pseudomonoids in $\Span$ correspond, respectively, to paracyclic sets and $\Gamma$-sets satisfying the $2$-Segal conditions. These results connect closely with work of the third author on $A_\infty$ algebras in $\infty$-categories of spans, as well as the growing body of work on higher Segal objects.
        Because our motivation comes from symplectic geometry and topological field theory, we emphasize the direct and computational nature of the classifications and their proofs.  
	\end{abstract}
	\tableofcontents
	
\section*{Introduction}
\addcontentsline{toc}{section}{Introduction}

The present paper is part of a program aimed at understanding Frobenius algebras and topological field theories valued in categories of Lagrangian correspondences. There are two main constructions of such Lagrangian correspondence categories: the Wehrheim-Woodward construction of \cite{WW} and \cite{LBW}, which is more $1$- and $2$-categorical in flavor, and the derived category of Lagrangian correspondences (see, e.g., \cite{CALAQUE2014926} and \cite{haugseng_iterated_2018}), which is fundamentally $\infty$-categorical. Because (higher) categories of spans play a significant role in both of these constructions, such span categories provide a common framework that can shed light on both approaches.

In this paper, we consider the bicategory of spans of sets, which we simply denote as $\Span$. In this setting, one can define a pseudomonoid, which is a coherent version of a monoid or algebra. As a special case of a result of one of the authors \cite{Stern:span}, pseudomonoids in $\Span$ are in correspondence with simplicial sets satisfying the \emph{$2$-Segal conditions} \cite{Dyckerhoff-Kapranov:Higher, GKT1}. Our main results are that a Frobenius structure on a pseudomonoid in $\Span$ corresponds to a paracyclic structure on the corresponding simplicial set, and that a commutative structure on a pseudomonoid in $\Span$ corresponds to a $\Gamma$-structure on the corresponding simplicial set.

These results could be generalized in various ways. However, since the eventual aim of this project is to draw together threads from classical mechanics, symplectic geometry, topological field theory, and higher category theory, we hope that it will be of interest to a broad mathematical audience, whose members may not be versed in all of the techniques of these disparate fields. Working in the context of spans of sets allows us to describe the various concepts and prove the results in an accessible  setting. In particular, a large part of the paper is devoted to reviewing all the concepts mentioned in the previous paragraph, making the paper relatively self-contained.

In the following, we give an extended introduction with more details and motivation.

\subsection*{Spans and categorification} 

Partially defined functions, unbounded operators, cobordisms, and (derived) Lagrangian correspondences all share a common feature: they are examples of spans in some category or $\infty$-category. A \emph{span} in a category or $\infty$-category $\scr{C}$ is a diagram 
\[
\begin{tikzcd}
	 & F\arrow[dl,"s"']\arrow[dr,"t"] & \\
	 X & & Y
\end{tikzcd}
\]
in $\scr{C}$. If $\scr{C}$ has finite limits, then spans can be composed, so a span as above can be interpreted as a morphism from $X$ to $Y$ in some appropriate (higher) category. In particular, one can consider the following:
\begin{itemize}
    \item The ($1$-)category for which the objects are the objects of $\scr{C}$, and the morphisms are isomorphism classes of spans in $\scr{C}$. We will denote this category as $\Span_1(\scr{C})$.
    \item The bicategory for which the objects are the objects of $\scr{C}$, the morphisms are spans in $\scr{C}$, and the $2$-morphisms are morphisms of spans. We will denote this bicategory as $\Span(\scr{C})$. This bicategory (with $\scr{C} = \Set$) will be the main focus of our paper, and it is described in more detail in Section \ref{sec:span}.
    \item When $\scr{C}$ is an $\infty$-category, a coherent version of the construction of the bicategory of spans yields an $(\infty,2)$-category. We will use the same notation as in the case where $\scr{C}$ is a 1-category, i.e., $\Span(\scr{C})$, to denote this $(\infty,2)$-category. In practice, many of the results we refer to about this $(\infty,2)$-category are really about its underlying $(\infty,1)$-category.   
\end{itemize}
We note that $\Span_1(\scr{C})$ can be recovered from $\Span(\scr{C})$ as the homotopy $1$-category.

Such span categories have been the subject of much interest in recent decades. Motivated by mathematical physics, Baez's \emph{groupoidification program} (see \cite{BaezGroupoid}) sought to understand linear algebra in terms of spans of groupoids. In that article, Baez, Hoffnung, and Walker, following ideas sketched in \cite{BaezDolan:fromfinitesets}, recast creation and annihilation operators on Fock space, Feynman diagrams, and various combinatorial algebras in terms of spans of groupoids. 

The threads introduced in \cite{BaezGroupoid} have also been taken up in the higher categorical setting. For example, the more recent \emph{homotopy linear algebra} \cite{GKT_HLA} of G\'alv\'ez-Carrillo, Kock, and Tonks extended the categorification into the realm of $\infty$-groupoids. Their work was aimed at studying categorifications of combinatorial (co)algebras, but otherwise bears close relations to that of Baez, Hoffnung, and Walker. 

In both \cite{BaezGroupoid} and \cite{GKT_HLA}, the key principle is that for certain sufficiently well-behaved spans 
\[
\begin{tikzcd}
	X  & F\arrow[l,"s"'] \arrow[r,"t"] & Y
\end{tikzcd}
\]
of sets, groupoids, or spaces, one can obtain a linear map of free vector spaces
\[
\begin{tikzcd}
	{\mathbb{Q}[\pi_0(X)]} \arrow[r]  & {\mathbb{Q}[\pi_0(Y)]}
\end{tikzcd}
\]
by a certain push-pull operation. The matrix elements of the linear maps are given by set, groupoid, or $\infty$-groupoid cardinalities of the appropriate fibers of $(s,t):F\to X\times Y$. This construction allows one to lift the structure maps defining (co)algebras in vector spaces to spans, and study the categorified algebraic structures in the appropriate span category. 

\subsection*{Algebraic structures in spans and 2-Segal objects}

One key to unraveling the structure of algebras of spans is their close relation to simplicial objects 
in $\scr{C}$ satisfying the \emph{2-Segal conditions}, which were introduced independently by Dyckerhoff and Kapranov in \cite{Dyckerhoff-Kapranov:Higher} and G\'alv\'ez-Carrillo, Kock, and Tonks in \cite{GKT1} (under the name \emph{decomposition spaces}). An associative algebra in $\Span(\scr{C})$ consists of an object $X_1\in \scr{C}$ which carries the algebraic structure, together with a \emph{unit} 
\[
\begin{tikzcd}
	\pt  & X_0 \arrow[l]\arrow[r,"s_0"] & X_1
\end{tikzcd}
\]
and $n$-fold multiplications 
\[
\begin{tikzcd}
	\underbrace{X_1\times \cdots\times X_1}_{n} & X_n \arrow[l]\arrow[r] & X_1 
\end{tikzcd}
\]
satisfying coherence conditions which encode unitality and associativity. From these data, one can extract a simplicial object 
\[
\begin{tikzcd}
	\cdots X_2 \arrow[r,shift right=0.3cm]\arrow[r,shift left=0.3cm] \arrow[r]  & X_1 \arrow[r,shift right=0.15cm]\arrow[r, shift left=0.15cm] \arrow[l,shift right=0.15cm]\arrow[l,shift left=0.15cm] & X_0 \arrow[l]
\end{tikzcd}
\]
in $\scr{C}$. The associativity and unitality conditions then require that the squares 
\[
\begin{tikzcd}
	X_n \arrow[r]\arrow[d] & X_{i,\ldots, j}\arrow[d]\\
	X_{0,\ldots,i,j,\ldots, n}\arrow[r] & X_{i,j}
\end{tikzcd} \quad \text{and}\quad \begin{tikzcd}
X_n \arrow[r,"s_i"]\arrow[d] & X_{n+1}\arrow[d,"{\{i,i+1\}}"]\\
X_0\arrow[r,"s_0"'] & X_1
\end{tikzcd}
\]
are pullback squares in $\scr{C}$. These are the 2-Segal conditions of \cite{Dyckerhoff-Kapranov:Higher}.\footnote{Technically, only the associativity conditions --- those encoded by the the left-hand squares --- are the 2-Segal conditions as defined in \cite{Dyckerhoff-Kapranov:Higher}, but the result of \cite{FGKW:unital} shows that, remarkably, in this case associativity implies unitality, and so the right-hand squares follow.}

Both groups of authors in \cite{Dyckerhoff-Kapranov:Higher} and \cite{GKT1} noted that there was a natural construction that yielded an algebra in $\Span(\scr{C})$ from a given 2-Segal simplicial object in $\scr{C}$. In \cite{Stern:span}, the last-named author showed that this construction was, in fact, an equivalence of $\infty$-categories, with suitable restrictions placed on the 1-morphisms.  In a similar vein, the first- and second-named authors classified algebras in $\Span_1(\Set)$ 
in terms of $2$-coskeletal simplicial sets satisfying weaker versions of the $2$-Segal conditions \cite{ContrerasMehtaSpan}. 

The present paper is aimed at expanding these results to algebras with additional structure: commutative ($E_\infty$) algebras and Frobenius algebras. While we focus exclusively on the case $\scr{C} = \Set$ throughout this paper, many of our results hold in greater generality. We also focus here on bijections of equivalence classes of objects, rather than equivalences of categories. This focus allows for more explicit constructions that should make the work accessible to a wider audience. In future work, we will prove higher-categorical analogues of the main results here. 

\subsection*{Topological field theories}

In addition to the motivation coming from symplectic geometry, the work here is partially motivated by a desire to study 2-dimensional topological field theories and 3-2-1 topological field theories. In the long term we will, in fact, study topological field theories valued in categories of Lagrangian correspondences, uniting the two motivations we have mentioned. 

A 2-dimensional closed (oriented) topological field theory consists of a symmetric monoidal functor 
\[
\begin{tikzcd}
	\on{Bord}_2^{\on{or}} \arrow[r] & \scr{C} 
\end{tikzcd}
\]
out of the category $\on{Bord}_2^{\on{or}}$ whose objects are disjoint unions of circles. Classically, the morphisms of $\on{Bord}_2^{\on{or}}$ are boundary-preserving diffeomorphism classes of (oriented) cobordisms. There are, however, variants of this category which remember more of the structure. In the realm of 2-categories, one can, instead, define a monoidal \emph{bicategory} $\scr{B}\!\on{ord}_2^{\on{or}}$ whose 2-morphisms are the isotopy classes of boundary-preserving diffeomorphisms, so that $\scr{B}\!\on{ord}_2^{\on{or}}$ remembers the mapping class groups of the cobordisms. Pushing these upward extensions to their logical conclusion, one could instead define an $(\infty,1)$-category, in which all of the higher isotopies are also encoded. 

2-categorical and $(\infty,1)$-categorical topological field theories are desirable inasmuch as the corresponding surface invariants come equipped with (possiby coherent) actions of the mapping class groups. However, they are often quite difficult to construct and classify. In the 1-categorical setting, functors out of $\on{Bord}_2^{\on{or}}$ to a symmetric monoidal 1-category $\scr{C}$ are known to correspond to commutative Frobenius algebras in $\scr{C}$. As such, the results of \cite{ContrerasMehtaSpan} are aimed at providing a classification of closed 2-dimensional topological field theories in the homotopy 1-category of spans of sets. 

No such classification is known for the 2-categorical and $(\infty,1)$-categorical cases. In the 2-categorical case, it should be possible to give classifying data in terms of generating 1- and 2-morphisms and relations, though to the best of our knowledge, this has not yet been accomplished. However, underlying a symmetric monoidal 2-functor $Z:\scr{B}\!\on{ord}_2^{\on{or}}\to \scr{C}$, there should be the following data:
\begin{itemize}
	\item A coherently commutative algebra in $\scr{C}$, obtained by restricting $Z$ to many-legged pairs of pants. 
	\item A homotopically non-degenerate trace $Z(S^1)\to Z(\varnothing)$, given by applying $Z$ to the cap.  
\end{itemize}
This means that, even in the absence of a full classification, a higher-categorical closed 2-dimensional topological field theory must contain a categorification of a commutative Frobenius algebra. A formal reflection of the presence of this data is implicit in \cite{CraneYetter}, where the authors explore coherent algebraic structures which are present in the data of a 3-2-1 topological field theory. 

In this work, we classify Frobenius pseudomonoids and commutative ($E_\infty$) pseudomonoids in the bicategory of spans of sets, effectively laying the groundwork for later work to construct 2-categorical topological field theories valued in spans. We defer a systematic treatment of such constructions, and of related state-sum constructions, to later work.  

\subsection*{Main results}

The two main results of this paper are classifications of algebraic structures in spans. Let $\Span = \Span(\Set)$. In keeping with traditional bicategorical terminology, we term what might otherwise be called an associative algebra in $\Span$ a \emph{pseudomonoid} in $\Span$. 

Our starting point in the paper is the following corollary of \cite[Theorem 2.25]{Stern:span}

\begin{thm*}
	There is a bijection between equivalence classes of pseudomonoids in $\Span$ and isomorphism classes of 2-Segal simplicial sets $X:\Delta^\op\to \Set$.
\end{thm*}

A pseudomonoid $X$  in a bicategory $\scr{B}$ equipped with a non-degenerate trace from $X$ to the monoidal unit $I$ is termed a \emph{Frobenius pseudomonoid}. Our first main result extends the classification of the theorem above to Frobenius pseudomonoids. To do so, we must replace the simplex category $\Delta$ with the \emph{paracyclic category} $\Lambda_{\infty}$. A paracyclic set, i.e., a functor from $\Lambda_{\infty}^\op \to \Set$, consists of a simplicial set $X$ together with $\mathbb{Z}$-actions on each $X_n$ compatible with the simplicial maps. Such a paracylic set is called 2-Segal when the underlying simplicial set is 2-Segal. Our first main result is then

\begin{thm*}[Theorem \ref{thm:paracyclic}]
	The bijection between pseudomonoids in $\Span$ and 2-Segal simplicial sets lifts to a bijection between equivalence classes of Frobenius pseudomonoids in $\Span$ and 2-Segal paracyclic sets $X:\Lambda^\op\to \Set$. 
\end{thm*}

Similarly, a pseudomonoid equipped with symmetry 2-isomorphism relating the composites of the multiplication with permutations to one another is called a \emph{commutative pseudomonoid}. To classify these, we must consider the category $\Phi_\ast$ of finite pointed cardinals, which admits a bijective-on-objects functor $\Cut:\Delta^\op \to \Phi_\ast$. A functor $X:\Phi_\ast\to \Set$ is then said to be 2-Segal if the simplicial set $X\circ \Cut$ is 2-Segal. 

\begin{thm*}[Theorem \ref{thm:equiv_phistar_comm}]
	The bijection between pseudomonoids in $\Span$ and 2-Segal simplicial sets lifts to a bijection between equivalence classes of commutative pseudomonoids and isomorphism classes of 2-Segal functors $X:\Phi_\ast\to \Set$. 
\end{thm*}

\subsection*{Future directions}
One of the expected applications of Theorem \ref{thm:paracyclic} is in the higher categorical interpretation of the symplectic category and its connections to TQFT. In particular, we intend to describe the Wehrheim-Woodward construction \cite{WW} of the symplectic category and algebraic structures therein in terms of 2-Segal simplicial objects. This approach differs from the derived version of the symplectic category in AKSZ theories \cite{CALAQUE2014926}, but we plan to relate the two descriptions in specific cases, e.g. when the target category is linear.

In a related direction, we plan to obtain (coherent) state sum models for open--closed TQFTS with values in the bicategory Span, following \cite{Lauda-Pfeiffer}.
In a work in progress, we prove that a state sum construction arising from a Frobenius pseudomonoid in Span satisfies the $(2,2)$- and, under certain conditions, the $(3,1)$-Pachner moves, which proves that in those cases, the partition function is independent of the triangulation of the corresponding surface. Additionally, we plan to lift 3-2-1 TFTs and related algebraic structures valued in Kapranov-Voevodsky 2-vector spaces to spans of groupoids (c.f. e.g. \cite{Morton}), providing a conceptual and categorical generalization of some of the ideas in \cite{CraneYetter}.

\subsection*{Structure of the paper}

In Section \ref{sec:span}, we review the definition of the bicategory $\Span$, as well as its symmetric monoidal structure In Section \ref{sec:pseudo}, we review the definitions of Frobenius and commutative pseudomonoid in a (symmetric) monoidal bicategory.

In Section \ref{sec:2segalpseudo}, we review the basics of $2$-Segal sets and sketch the correspondence in \cite{Stern:span} between $2$-Segal sets and pseudomonoids in $\Span$. In particular, in Section \ref{sec:faceanddegen}, we describe a graphical calculus for face and degeneracy maps in $2$-Segal sets that is used extensively in the proofs of our main results.

In Section \ref{sec:paracyclicfrobenius}, we review paracyclic structures and then prove our first main result, that Frobenius pseudomonoids in $\Span$ correspond to $2$-Segal paracyclic sets.

In Section \ref{sec:gammacommutative}, we review $\Gamma$-structures and then prove our second main result, that commutative pseudomonoids in $\Span$ correspond to $2$-Segal $\Gamma$-sets.

\subsection*{Acknowledgements}
We thank the anonymous referee for useful comments and suggestions.
R.M. thanks the University of Illinois, Urbana-Champaign for hospitality during a sabbatical when much of this work was completed. WHS thanks Julie Bergner for useful conversations about 2-Segal sets, and wishes to acknowledge the support of the NSF Research Training Group at the University of Virginia (grant number DMS-1839968) during the preparation of this work. I.C. thanks the Amherst College Provost and Dean of the Faculty’s Research Fellowship (2021-2022).

\section{The bicategory of spans}\label{sec:span}

In this section, we briefly review the bicategory of spans and its symmetric monoidal structure. This isn't intended to be a thorough treatment, but only a sufficient amount of information required to understand the definitions in Section \ref{sec:pseudo}. For more details, we refer the reader to \cite{CKWW} for a construction of the symmetric monoidal bicategory of spans in a category with finite limits, and to \cite{stay}, where a clear and complete definition of symmetric monoidal bicategories is given, and where bicategories of spans appear as an example. We also refer to \cite{ahmadi}, which gives a nice overview, eliminating some redundant conditions in \cite{stay} (but doesn't cover symmetric structures), and \cite{Johnson-Yau}, where the construction of the bicategory of spans is explicitly described in Example 2.1.22.

\begin{rmk}[Bicategorical conventions]
	Before proceeding further into our discussion of spans, let us fix some notational conventions for bicategories which we will use in the remainder of the paper. In a bicategory $\scr{B}$, we will denote by $\id_X$ (or sometimes simply $\id$, where the object is clear from context) the identity 1-morphism on an object $X$. We will denote by $\on{Id}_f$ (or, as before, simply $\on{Id}$) the identity 2-morphism on a 1-morphism $f$. In the case of a monoidal bicategory $\scr{B}$, we will use $I$ to denote the monoidal unit. 
\end{rmk}

The bicategory $\Span$ is given by the following data:
\begin{itemize}
    \item objects are sets;
    \item a morphism from a set $X$ to a set $Y$ is a span $X \xleftarrow{f_1} A \xrightarrow{f_2} Y$, which we will sometimes also denote as $f: X \leftarrow A \rightarrow Y$;
    \item a $2$-morphism from $X \xleftarrow{f_1} A \xrightarrow{f_2} Y$ to $X \xleftarrow{g_1} B \xrightarrow{g_2} Y$ is a map $h: A \to B$ such that $g_i h = f_i$;
    \item the identity morphism is the span $\id_X: X \xleftarrow{\id} X \xrightarrow{\id} X$;
    \item the composition of $f: X \xleftarrow{f_1} A \xrightarrow{f_2} Y$ with $g: Y \xleftarrow{g_1} B \xrightarrow{g_2} Z$ is $g \circ f: X \xleftarrow{f_1 p_1} A \times_Y B \xrightarrow{g_2 p_2} Z$, where
    \[ A \times_Y B = \{(a,b) \in A \times B \suchthat f_2(a) = g_1(b)\}\]
    and $p_i$ are the projection maps onto the two components.
\end{itemize}
The remaining bicategorical data (horizontal composition of $2$-morphisms, associator, and unitors) arise naturally from the universal property of pullbacks. In particular, the associator is given by the canonical rebracketing isomorphism 
\[
(A \times_X B) \times_Y C \cong A \times_X (B \times_Y C),
\]
and the unitors are given by the canonical isomorphisms $X \times_X A \cong A \cong A \times_X X$. We will generally use these isomorphisms freely without making them explicit.

\begin{rmk}
	Throughout this paper, we will depict 1-morphisms in the monoidal bicategory $\Span$ with the aid of a variant on string diagrams. Since there are many conventions for diagrammatically depicting morphisms in various flavors of higher category, let us briefly comment on our conventions here. In general, our pictures will be of the form described in \cite[\S 2.3]{Baez2011} for monoidal 1-categories. While this convention necessitates some imprecision (as we will discuss below), we feel that this graphical representation is best suited to quickly connecting the reader's intuition to the technical definitions given.  
	\begin{itemize}
		\item We will draw a 1-morphism $f$ from $X$ to $Y$ in as a labeled strand 
		 \stringdiagram{
		 	\path (0,3.5) node {$X$};
		 	\identity{0}{2}
		 	\morphism{0}{0}{f}
		 	\identity{0}{-2}
		 	\path (0,-3.5) node {$Y$};
		}
		read from top to bottom. In most cases, where it is clear from context what the objects $X$ and $Y$ are, we will omit the corresponding labels. 
		\item Compositions will be depicted by vertical concatenation of diagrams, so that, for instance 
		\stringdiagram{
		\morphism{0}{0}{f}
		\morphism{0}{-2}{g}
		}
		represents the composite morphism $g\circ f$. 
		\item We will depict monoidal products by drawing two strands next to one another, so, for instance, the diagram
		\stringdiagram{
		\identity{0}{2}
		\morphism{0}{0}{f}
		\identity{0}{-2} 
		 \draw (1,0.6) rectangle (3,-0.6);
		 \path (2,0) node {$m$}; 
		 \draw (1.4,0.6) -- (1.4,3);
		 \draw (2.6,0.6) -- (2.6,3);
		 \draw (2,-0.6) -- (2,-3); 
		 \path (0,3.5) node {$X$}; 
		 \path (1.4,3.5) node {$Y$}; 
		 \path (2.6,3.5) node {$Z$}; 
		 \path (0,-3.5) node {$U$};
		 \path (2,-3.5) node {$V$};  
			}
		represents $f\otimes m :X\otimes (Y\otimes Z)\to U\otimes V$. 
	\end{itemize}  
	As mentioned above, when implemented in a monoidal bicategory, rather than a monoidal 1-category, this method of drawing diagrams involves some imprecisions. In particular:
	\begin{itemize}
		\item The composition and is not strictly associative, but rather has associativity \emph{data} given by 2-isomorphisms. This means that diagrams involving, e.g., a 2-fold composition do not yield a single, well-defined 1-morphism. However, the coherence of bicategories means that, given two ways of reading a given diagram, there is a unique 2-morphism built from associators which relates the two resulting 1-morphisms. Since the associators in $\Span$ are simply rebracketing isomorphisms, we feel that this does not substantially detract from the understandability of the diagrammatics. 
		\item Nearly identical considerations to those in the previous point apply to the monoidal product. Once again, associativity only holds up to rebracketing isomorphisms, but this does not substantially hinder the reader in understanding the diagrammatics.  
	\end{itemize} 
	One final word of warning for those already familiar with string diagrams for bicategories. One common convention is to draw pictures in which objects are represented by regions, morphisms by lines and 2-morphisms by vertices. We \emph{do not} make use of this pictorial formalism in this paper. 
\end{rmk}

The symmetric monoidal structure on $\Span$ comes from the Cartesian product. We highlight the data that appears in the definitions in Section \ref{sec:pseudo}.

\begin{itemize}
    \item Part of the data of a monoidal bicategory is what Stay \cite{stay} calls the \emph{tensorator}, which is an invertible $2$-morphism controlling the failure of the monoidal product to commute with composition. In the case of $\Span$, it has the following form. Given two pairs of composable spans $X \leftarrow A \rightarrow X' \leftarrow A' \rightarrow X''$ and $Y \leftarrow B \rightarrow Y' \leftarrow B' \rightarrow Y''$, the tensorator is the canonical isomorphism
    \[ (A \times B) \times_{X' \times Y'} (A' \times B') \cong (A \times_{X'} A') \times (B \times_{Y'} B'),\]
    arising from the universal property of pullbacks, relating the composition of products to the product of compositions.

    In particular, given spans $f: X \leftarrow A \rightarrow X'$ and $g: Y \leftarrow B \rightarrow Y'$, the tensorator and unitors give canonical isomorphisms
    \[ (A \times Y) \times_{X' \times Y} (X' \times B) \cong  (A \times_{X'} X') \times (Y \times_Y B) \cong A \times B\]
    and
    \[ (X \times B) \times_{X \times Y'} (A \times Y') \cong (X \times_X A) \times (B \times_{Y'} Y') \cong A \times B.\]
    Composing these isomorphisms gives an invertible $2$-morphism
    \[ (\id_{X'} \times g) \circ (f \times \id_Y) \xRightarrow{c_{f,g}} (f \times \id_{Y'}) \circ (\id_X \times g).\]
    Graphically, $c_{f,g}$ implements the ``slide move'' shown by the following string diagrams:
    \stringdiagram{
    \morphism{0}{1}{f}
    \identity{1}{1}
    \identity{0}{-1}
    \morphism{1}{-1}{g}
    \nattrans{3}{0}{c_{f,g}}
    \morphism{5}{-1}{f}
    \identity{6}{-1}
    \identity{5}{1}
    \morphism{6}{1}{g}
    }
    \item The braiding data in a braided monoidal bicategory includes, for every pair of objects $X,Y$, a morphism $\rho_{X,Y}: X \otimes Y \to Y \otimes X$ and, for every pair of morphisms $f:X \to X'$, $g: Y \to Y'$, a $2$-morphism $\rho_{f,g}: \rho_{X',Y'} \circ (f \otimes g) \Rightarrow (g \otimes f) \circ \rho_{X,Y}$.

    In $\Span$, we take $\rho_{X,Y}$ to be the span $X \times Y \xleftarrow{1} X \times Y \xrightarrow{\tilde{\rho}} Y \times X$, where $\tilde{\rho}$ is the map that exchanges the components. Given spans $f: X \leftarrow A \rightarrow X'$ and $g: Y \leftarrow B \rightarrow Y'$, the $2$-morphism $\rho_{f,g}$ is given by the isomorphism
    \[ (A \times B) \times_{X' \times Y'} (X' \times Y') \cong A \times B \cong (X \times Y) \times_{Y \times X} (B \times A),\]
    where on the left the projection of $A \times B$ onto the first component is the identity, and on the right the projection of $A \times B$ onto the second component is $\tilde{\rho}$.

    Graphically, we represent $\rho_{X,Y}$ as a crossing of strings
    \stringdiagram{\swap{0}{0}}
    and then $\rho_{f,g}$ implements the following move:
    \stringdiagram{
    \morphism{0}{1}{f}
    \morphism{1}{1}{g}
    \slimswap{.5}{-1}
    \nattrans{3}{0}{\rho_{f,g}}
    \slimswap{5.5}{1}
    \morphism{5}{-1}{g}
    \morphism{6}{-1}{f}
    }
    
    \item The braiding data in a braided monoidal bicategory also includes, for objects $X,Y,Z$, an invertible $2$-morphism 
    \begin{equation}\label{eqn:R}
    (\id_Y \otimes \rho_{X,Z}) \circ \alpha_{Y,X,Z} \circ (\rho_{X,Y} \otimes \id_Z) \xRightarrow{R_{X|YZ}} \alpha_{Y,Z,X} \circ \rho_{X,Y\otimes Z} \circ \alpha_{X,Y,Z},
    \end{equation}
    called the \emph{(left) hexagonator}. Here, $\alpha_{X,Y,Z}: (X \otimes Y) \otimes Z \to X \otimes (Y \otimes Z)$ is the associator for the monoidal structure.

    In $\Span$, we take $\alpha_{X,Y,Z}$ to be the span $(X \times Y) \times Z \xleftarrow{1} (X \times Y) \times Z \xrightarrow{\tilde{\alpha}} X \times (Y \times Z)$, where $\tilde{\alpha}$ is the natural rebracketing map. Then $R_{X|YZ}$ is the isomorphism of spans that arises from the unique identification of the compositions on both sides of \eqref{eqn:R} with $(X \times Y) \times Z \xleftarrow{1} (X \times Y) \times Z \rightarrow Y \times(Z \times X)$.

    In the graphical notation, the associators do not explicitly appear, which helps to see the true nature of the hexagonator as relating the two different ways to use braiding maps to go from $X \times Y \times Z$ to $Y \times Z \times X$: exchanging $X$ and $Y$ and then exchanging $X$ and $Z$, versus exchanging $X$ and $Y \times Z$:
    \stringdiagram{
    \slimswap{.5}{1}
    \identity{2}{1}
    \identity{0}{-1}
    \slimswap{1.5}{-1}
    \nattrans{5}{0}{R_{X|YZ}}
    \swaptwoc{8.5}{0}
    }

    \item The data of a symmetric monoidal bicategory\footnote{The definition of commutative pseudomonoid only requires a sylleptic monoidal structure. However, since symmetry is a property, not requiring additional data, and since $\Span$ satisfies this property, there is no need in this paper for us to consider structures that are sylleptic but not symmetric.}  includes, for objects $X,Y$, an invertible $2$-morphism
    \[ \rho_{Y,X} \circ \rho_{X,Y} \xRightarrow{v_{X,Y}} \id_{X\times Y},\]
    called the \emph{syllepsis}.

    In $\Span$, $v_{X,Y}$ is the unique isomorphism from the composition $X \times Y \xleftarrow{1} X \times Y \xrightarrow{\tilde{\rho}} Y \times X \xleftarrow{1} Y \times X \xrightarrow{\tilde{\rho}} X \times Y$ to the identity morphism on $X \times Y$.

    Graphically, we have
    \stringdiagram{
    \slimswap{.5}{1}
    \slimswap{.5}{-1}
    \nattrans{3}{0}{v_{X,Y}}
    \identity{5}{0}
    \identity{6}{0}
    }
\end{itemize}

\section{Pseudomonoids with Frobenius and commutative structures}\label{sec:pseudo}

In this section, we review definitions of Frobenius and commutative pseudomonoids. The notion of pseudomonoid dates back to \cite{DayStreet}. As is the case for Frobenius algebras, there are several different essentially equivalent definitions of Frobenius pseudomonoid. The definition we use coincides with that of \cite{Street:Frobenius}, which seems to be the simplest possible approach. We also mention \cite{Lauda:Frobenius} and \cite{DunnVicary}, which use an essentially identical definition with graphical presentations of the coherence conditions. A nice reference for commutative pseudomonoids is \cite{Verdon} (where the term ``symmetric'' is used instead of ``commutative''), which includes string diagrams that are similar in style to those here. 

\subsection{Pseudomonoids}

Let $\C$ be a monoidal bicategory. A \emph{pseudomonoid} in $\C$ consists of
\begin{itemize}
    \item an object $X$,
    \item morphisms $\eta: I \to X$ (unit) and $\mu: X \otimes X \to X$ (multiplication),
    \item invertible $2$-morphisms $a: \mu \circ (\mu \otimes \id_X) \Rightarrow \mu \circ (\id_X \otimes \mu)$ (associator), $\ell: \mu \circ (\eta \otimes \id_X) \Rightarrow \id_X$ (left unitor), and $r: \mu \circ (\id_X \otimes \eta) \Rightarrow \id_X$ (right unitor),
\end{itemize}
satisfying two coherence conditions. The coherence conditions can be described using string diagrams, as follows. We represent $\eta$ and $\mu$, respectively, by the following diagrams (read top to bottom).
\stringdiagram{
\begin{scope}
    \unit{0}{0}
\end{scope}
\begin{scope}[shift={(6,0)}]
    \multiplication{0}{0}
\end{scope}
}
Note that we here depict the monoidal unit as a black circle, and do not label the strings. 

With these definitions, $a$, $\ell$ and $r$ can be drawn as morphisms of diagrams:
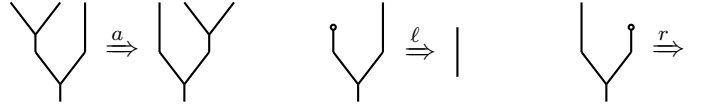
\stringdiagram{
\begin{scope} 
\multiplication{0}{1}
\identity{2}{1}
\multiplication{1}{-1}
\nattrans{3.5}{0}{a}
\identity{5}{1}
\multiplication{7}{1}
\multiplication{6}{-1}
\end{scope}

\begin{scope}[shift={(12,0)}]
\unit{0}{1}
\identity{2}{1}
\multiplication{1}{-1}
\nattrans{3.5}{0}{\ell}
\identity{5}{0}
\end{scope}

\begin{scope}[shift={(22,0)}]
\unit{2}{1}
\identity{0}{1}
\multiplication{1}{-1}
\nattrans{3.5}{0}{r}
\identity{5}{0}
\end{scope}
}
The coherence conditions consist of the \emph{triangle equation} 

	\stringdiagramlabel{
		\begin{scope}
			\begin{scope}
				\identity{0}{3}
				\unit{2}{3}
				\identity{3}{3}
				\multiplication{1}{1}
				\identity{3}{1}
				\multiplication{2}{-1}
				\highlight{-.2}{-2.2}{3.2}{2.2}
				\nattrans{4.5}{1}{a}
			\end{scope}
			\begin{scope}[shift={(6,0)}]
				\identity{0}{3}
				\unit{1}{3}
				\identity{3}{3}
				\multiplication{2}{1}
				\identity{0}{1}
				\multiplication{1}{-1}
				\highlight{.8}{.2}{3.2}{4.2}
				\nattrans{4.5}{1}{\ell}
			\end{scope}
			\begin{scope}[shift={(12,0)}]
				\multiplication{1}{1}
			\end{scope}
			\enclose{-.5}{-2.5}{14.5}{4.5}
		\end{scope}
		\equals{16}{1}
		\begin{scope}[shift={(18,0)}]
			\begin{scope}
				\identity{0}{3}
				\unit{2}{3}
				\identity{3}{3}
				\multiplication{1}{1}
				\identity{3}{1}
				\multiplication{2}{-1}
				\highlight{-.2}{.2}{2.2}{4.2}
				\nattrans{4.5}{1}{r}
			\end{scope}
			\begin{scope}[shift={(6,0)}]
				\multiplication{1}{1}
			\end{scope}
			\enclose{-.5}{-2.5}{8.5}{4.5}
		\end{scope}
	}{eqn:triangle}
and the \emph{pentagon equation}
\stringdiagramlabel{
\begin{scope}[scale=0.9]
\begin{scope}
    \begin{scope}
        \slimmultiplication{0}{5}
        \identity{2}{5}
        \identity{3}{5}
        \multiplication{1}{3}
        \identity{3}{3}
        \multiplication{2}{1}
        \nattrans{4.5}{3}{a}
        \highlight{-0.6}{2.2}{2.2}{6.2}
    \end{scope}
    \begin{scope}[shift={(6,0)}]
        \slimmultiplication{2}{5}
        \identity{0}{5}
        \identity{3}{5}
        \multiplication{1}{3}
        \identity{3}{3}
        \multiplication{2}{1}
        \nattrans{4.5}{3}{a}
        \highlight{-0.2}{-0.2}{3.2}{4.2}
    \end{scope}
    \begin{scope}[shift={(12,0)}]
        \slimmultiplication{1}{5}
        \identity{0}{5}
        \identity{3}{5}
        \multiplication{2}{3}
        \identity{0}{3}
        \multiplication{1}{1}
        \nattrans{4.5}{3}{a}
        \highlight{0.4}{1.8}{3.2}{6.2}
    \end{scope}
    \begin{scope}[shift={(18,0)}]
        \slimmultiplication{3}{5}
        \identity{0}{5}
        \identity{1}{5}
        \multiplication{2}{3}
        \identity{0}{3}
        \multiplication{1}{1}
    \end{scope}  
    \enclose{-1}{-.5}{22}{6.5}
\end{scope}
\equals{23}{3}
\begin{scope}[shift={(25,0)}]
    \begin{scope}
        \slimmultiplication{0}{5}
        \identity{2}{5}
        \identity{3}{5}
        \multiplication{1}{3}
        \identity{3}{3}
        \multiplication{2}{1}
        \nattrans{4.5}{3}{a}
        \highlight{-0.2}{-0.2}{3.2}{4.2}
    \end{scope}
    \begin{scope}[shift={(6,0)}]
        \slimmultiplication{0}{5}
        \identity{1}{5}
        \identity{3}{5}
        \multiplication{2}{3}
        \identity{0}{3}
        \multiplication{1}{1}
        \nattrans{4.5}{3}{c_{\mu,\mu}}
        \highlight{-0.6}{2.2}{3.2}{6.2}
    \end{scope}
    \begin{scope}[shift={(12,0)}]
        \slimmultiplication{3}{5}
        \identity{0}{5}
        \identity{2}{5}
        \multiplication{1}{3}
        \identity{3}{3}
        \multiplication{2}{1}
        \nattrans{4.5}{3}{a}
        \highlight{-0.2}{-0.2}{3.2}{4.2}
    \end{scope}
    \begin{scope}[shift={(18,0)}]
        \slimmultiplication{3}{5}
        \identity{0}{5}
        \identity{1}{5}
        \multiplication{2}{3}
        \identity{0}{3}
        \multiplication{1}{1}
    \end{scope}  
    \enclose{-1}{-.5}{22}{6.5}
\end{scope}
\end{scope}
}{eqn:pentagon}
In the above equations, the dotted boxes indicate the portion of the diagram where the $2$-morphism is applied.

We will see in Section 3.3 that this bicategorical manifestation of associativity is closely related to the \emph{associahedra}: polytopes introduced by Stasheff in \cite{Stasheff} which control coherent associativity for topological spaces. In particular, the pentagon equation can be reinterpreted as the commutativity of a diagram corresponding to the associahedron $K_4$, which is simply a pentagon (c.f. figure \ref{fig:assoc4}).

A reader who is unfamiliar with pseudomonoids may find the following examples to be helpful for developing intuition for the definition.

\begin{example}\label{ex:monoid1}
Let $\C$ be a monoidal category, viewed as a monoidal bicategory where the only $2$-morphisms are identities. Then the notion of pseudomonoid reduces to that of a monoid object in $\C$. In particular, the coherence conditions are trivial in this case.
\end{example}

\begin{example}\label{ex:moncat1}
Consider the bicategory $\Cat$, for which the objects are categories, the morphisms are functors, and the $2$-morphisms are natural transformations, equipped with the Cartesian (symmetric) monoidal structure. A pseudomonoid in $\Cat$ is the same thing as a monoidal category.
\end{example}

\subsection{Frobenius pseudomonoids}

Let $\C$ be a monoidal bicategory. A \emph{Frobenius pseudomonoid} in $\C$ is a pseudomonoid $X$ equipped with a \emph{counit} morphism $\varepsilon : X \to I$ such that the pairing $\alpha := \varepsilon \circ \mu : X \otimes X \to I$ is biexact, in the sense of \cite{Street:Frobenius}. Explicitly, biexactness means that there exists a \emph{copairing} $\beta: I \to X \otimes X$ and invertible $2$-morphisms $z: (\id_X \otimes \varepsilon) \circ (\id_X \otimes \mu) \circ (\beta \otimes \id_X) \Rightarrow \id_X$ and $n: (\varepsilon \otimes \id_X) \circ (\mu \otimes \id_X) \circ (\id_X \otimes \beta) \Rightarrow \id_X$ where, if we depict $\varepsilon$, $\alpha$, and $\beta$ respectively as
    \stringdiagram{
    \counit{0}{0}
    \begin{scope}[shift={(8,0)}]
        \pairing{0}{0}
        \equals{2}{0}
        \multiplication{4}{1}
        \counit{4}{-1}
    \end{scope}
    \copairing{20}{0}
    }
and show $z$ and $n$ as
    \stringdiagram{
    \begin{scope}
        \copairing{0}{1}
        \identity{3}{1}
        \identity{-1}{-1}
        \pairing{2}{-1}
        \nattrans{4.5}{0}{z}
        \identity{6}{0}
    \end{scope}
    \begin{scope}[shift={(12,0)}]
        \copairing{2}{1}
        \identity{-1}{1}
        \identity{3}{-1}
        \pairing{0}{-1}
        \nattrans{4.5}{0}{n}
        \identity{6}{0}
    \end{scope}
    }
then the following coherence conditions are satisfied:
\stringdiagram{
\begin{scope}
\begin{scope}
    \identity{0}{2}
    \copairing{3}{2}
    \identity{6}{2}
    \pairing{1}{0}
    \identity{4}{0}
    \identity{6}{0}
    \pairing{5}{-2}
    \nattrans{7.5}{0.5}{n}
    \highlight{-.2}{-.2}{4.2}{3.2}
\end{scope}
\begin{scope}[shift={(10,0)}]
    \pairing{0}{0}
\end{scope}
\enclose{-.5}{-2.5}{11.5}{3.5}
\end{scope}
\equals{13}{0}
\begin{scope}[shift={(15,0)}]
\begin{scope}
    \identity{0}{2}
    \copairing{3}{2}
    \identity{6}{2}
    \pairing{1}{0}
    \identity{4}{0}
    \identity{6}{0}
    \pairing{5}{-2}
    \nattrans{8}{0.5}{c_{\alpha, \alpha}}
    \highlight{-.2}{-2.2}{6.2}{1.2}
\end{scope}
\begin{scope}[shift={(10,0)}]
    \identity{0}{2}
    \copairing{3}{2}
    \identity{6}{2}
    \pairing{1}{-2}
    \identity{0}{0}
    \identity{2}{0}
    \pairing{5}{0}
    \nattrans{7.5}{0.5}{z}
    \highlight{1.8}{-.2}{6.2}{3.2}
\end{scope}
\begin{scope}[shift={(20,0)}]
    \pairing{0}{0}
\end{scope}
\enclose{-.5}{-2.5}{21.5}{3.5}
\end{scope}
}

\stringdiagram{
\begin{scope}
\begin{scope}
    \copairing{1}{2}
    \identity{0}{0}
    \identity{2}{0}
    \copairing{5}{0}
    \identity{0}{-2}
    \pairing{3}{-2}
    \identity{6}{-2}
    \nattrans{7.5}{-0.5}{n}
    \highlight{1.8}{-3.2}{6.2}{0.2}
\end{scope}
\begin{scope}[shift={(10,0)}]
    \copairing{0}{0}
\end{scope}
\enclose{-.5}{-3.5}{11.5}{2.5}
\end{scope}
\equals{13}{0}
\begin{scope}[shift={(15,0)}]
\begin{scope}
    \copairing{1}{2}
    \identity{0}{0}
    \identity{2}{0}
    \copairing{5}{0}
    \identity{0}{-2}
    \pairing{3}{-2}
    \identity{6}{-2}
    \nattrans{8}{-0.5}{c_{\beta,\beta}}
    \highlight{-.2}{-1.2}{6.2}{2.2}    
\end{scope}
\begin{scope}[shift={(10,0)}]
    \copairing{5}{2}
    \identity{4}{0}
    \identity{6}{0}
    \copairing{1}{0}
    \identity{0}{-2}
    \pairing{3}{-2}
    \identity{6}{-2}
    \nattrans{7.5}{-0.5}{z}
    \highlight{-0.2}{-3.2}{4.2}{0.2}    
\end{scope}
\begin{scope}[shift={(20,0)}]
    \copairing{0}{0}
\end{scope}
\enclose{-0.5}{-3.5}{21.5}{2.5}
\end{scope}
}
Given a pseudomonoid $X$, we consider two Frobenius structures on $X$ to be equivalent if their associated counit morphisms are $2$-isomorphic.

\begin{rmk}
    In \cite{DunnVicary,Lauda:Frobenius}, the copairing $\beta$ and the $2$-morphisms $z,n$ are included as part of the data of a Frobenius pseudomonoid. This is in contrast to the above definition, taken from \cite{Street:Frobenius}, where $\beta$ and $z,n$ are required to exist but are not specified. However, it can be shown that this data, if it exists, is unique up to coherent isomorphism. In this sense, the two versions of the definition are essentially equivalent.
\end{rmk}

\begin{example}
       Following up on Example \ref{ex:monoid1}, the notion of Frobenius pseudomonoid reduces to that of Frobenius object (see, for example, \cite{kock}) in the case where $\C$ is a monoidal category. In particular, when $\C = \Vect$ is the category of vector spaces with tensor product as the monoidal structure, we recover the notion of Frobenius algebra.
\end{example}

\begin{example}
    This is actually more of a nonexample, following up on Example \ref{ex:moncat1}. The only category that admits the structure of a Frobenius pseudomonoid in $\Cat$ is the terminal category with one object and only the identity morphism. 
    
    This fact is a categorification of the fact that the only Frobenius object in the category $\Set$ is the one-point set. One way to rectify this situation is to pass to $\Rel$, the category whose objects are sets and morphisms are relations. Frobenius objects in $\Rel$ have been studied in detail in \cite{CLMM, Mehta-Zhang}. In the categorified setting, the analogue is to consider Frobenius pseudomonoids in the bicategory of profunctors.
\end{example}

\subsection{Commutative pseudomonoids}\label{subsec:comm_pseudomonoids}

Let $\C$ be a symmetric monoidal bicategory. A \emph{commutative pseudomonoid} in $\C$ is a pseudomonoid $X$ equipped with an invertible $2$-morphism $\gamma: \mu \Rightarrow \mu \circ \rho_{X,X}$
satisfying the \emph{symmetry equation }
\stringdiagramlabel{
\begin{scope}
    \begin{scope}
        \multiplication{0}{0}
        \nattrans{2.5}{0}{\gamma}
    \end{scope}
    \begin{scope}[shift={(5,0)}]
        \swap{0}{1}
        \multiplication{0}{-1}
        \nattrans{2.5}{0}{\gamma}
        \highlight{-1.2}{-2.2}{1.2}{0.2}
    \end{scope}
    \begin{scope}[shift={(10,0)}]
        \swap{0}{2}
        \swap{0}{0}
        \multiplication{0}{-2}
    \end{scope}
    \enclose{-1.5}{-3.5}{11.5}{3.5}
\end{scope}
\equals{13}{0}
\begin{scope}[shift={(16,0)}]
    \begin{scope}
        \multiplication{0}{0}
        \nattrans{3}{0}{v_{X,X}^{-1}}
        \highlight{-1.2}{0.8}{1.2}{1.2}
    \end{scope}
    \begin{scope}[shift={(6,0)}]
        \swap{0}{2}
        \swap{0}{0}
        \multiplication{0}{-2}
    \end{scope}
    \enclose{-1.5}{-3.5}{7.5}{3.5}
\end{scope}
}{eqn:symmetry}
and the \emph{hexagon equation}
\stringdiagramlabel{
\begin{scope}
    \begin{scope} 
        \multiplication{0}{1}
        \identity{2}{1}
        \multiplication{1}{-1}
        \nattrans{3.5}{0}{a}
    \end{scope}
    \begin{scope}[shift={(5,0)}]
        \identity{0}{1}
        \multiplication{2}{1}
        \multiplication{1}{-1}
        \nattrans{3.5}{0}{\gamma}
        \highlight{-.2}{-2.2}{2.2}{0.2}
    \end{scope}
    \begin{scope}[shift={(10,0)}]
       \identity{0}{2}
        \multiplication{2}{2}
        \swap{1}{0}
        \multiplication{1}{-2}
        \nattrans{5}{0}{\rho_{\id_X,\mu}}
        \highlight{-.2}{-1.2}{3.2}{3.2}
    \end{scope}
    \begin{scope}[shift={(18,0)}]
        \swaptwoa{.5}{2}
        \multiplication{0}{0}
        \identity{2}{0}
        \multiplication{1}{-2}
        \nattrans{3}{0}{a}
        \highlight{-1.2}{-3.2}{2.2}{1.2}
    \end{scope}
    \begin{scope}[shift={(22,0)}]
        \swaptwob{.5}{2}
        \identity{0}{0}
        \multiplication{2}{0}
        \multiplication{1}{-2}
    \end{scope}
\enclose{-1.5}{-3.5}{25.5}{3.5}
\end{scope}
\equals{10}{-8}
\begin{scope}[shift={(13,-9)}]
    \begin{scope} 
        \multiplication{0}{1}
        \identity{2}{1}
        \multiplication{1}{-1}
        \nattrans{3.5}{0}{\gamma}
        \highlight{-1.2}{-.2}{1.2}{2.2}
    \end{scope}
    \begin{scope}[shift={(6,0)}]
        \swap{0}{2}
        \identity{2}{2}
        \multiplication{0}{0}
        \identity{2}{0}
        \multiplication{1}{-2}
        \nattrans{3.5}{0}{a}
        \highlight{-1.2}{-3.2}{2.2}{1.2}
    \end{scope}
    \begin{scope}[shift={(11,0)}]
        \slimswap{0.5}{2}
        \identity{3}{2}
        \multiplication{2}{0}
        \identity{0}{0}
        \multiplication{1}{-2}
        \nattrans{4.5}{0}{\gamma}
        \highlight{.8}{-1.2}{3.2}{1.2}
    \end{scope}
    \begin{scope}[shift={(17,0)}]
        \slimswap{0.5}{3}
        \identity{3}{3}
        \identity{0}{1}
        \swap{2}{1}
        \multiplication{2}{-1}
        \identity{0}{-1}
        \multiplication{1}{-3}
        \nattrans{5.5}{0}{R_{X|XX}}
        \highlight{-.2}{-.2}{3.2}{4.2}
    \end{scope}
    \begin{scope}[shift={(25,0)}]
        \swaptwob{.5}{2}
        \identity{0}{0}
        \multiplication{2}{0}
        \multiplication{1}{-2}
    \end{scope}
\enclose{-1.5}{-4.5}{28.5}{4.5}
\end{scope}
}{eqn:hexagon}

\begin{example}
    Following up on Example \ref{ex:monoid1}, the notion of commutative pseudomonoid reduces to that of commutative monoid object in the case where $\C$ is a symmetric monoidal category. 
\end{example}

\begin{example}
    Following up on Example \ref{ex:moncat1}, a commutative pseudomonoid in the symmetric monoidal bicategory $\Cat$ is the same thing as a symmetric monoidal category.
\end{example}

\section{\texorpdfstring{$2$}{2}-Segal sets and pseudomonoids in \texorpdfstring{$\Span$}{Span}}
\label{sec:2segalpseudo}

The notion of a \emph{$2$-Segal set} (and more generally, \emph{$2$-Segal space}) was introduced by Dyckerhoff and Kapranov \cite{Dyckerhoff-Kapranov:Higher} as a generalization of the Segal conditions that hold when a simplicial set is the nerve of a category (also see \cite{BOORS} for a nice exposition). The equivalent concept of \emph{decomposition space} was independently introduced by G\'alvez-Carillo, Kock, and Tonks \cite{GKT1}. 

In \cite{Stern:span}, it was shown that there is an $\infty$-categorical equivalence between $2$-Segal spaces and coherent algebra structures in $\infty$-categories of spans. It follows from this result that there is a correspondence between $2$-Segal sets and pseudomonoids in $\Span$.

In Section \ref{sec:2Segal}, we review the definition and some properties of $2$-Segal sets. In Section \ref{sec:faceanddegen}, we describe a simple but useful graphical calculus for visualizing face and degeneracy maps on $2$-Segal sets. In Section \ref{sec:pseudospan}, we sketch the correspondence between $2$-Segal sets and pseudomonoids in $\Span$. Because the proof of the correspondence in \cite{Stern:span} is $\infty$-categorical, it is difficult for an uninitiated reader to follow; we hope the more elementary description here, using an approach via dual graphs that originates in \cite{dyckerhoff-kapranov:crossed}, will be useful to such readers.

In Section \ref{sec:examples2segal}, we describe some examples of $2$-Segal sets. These are all examples that have appeared elsewhere in the literature (e.g.\ \cite{Dyckerhoff-Kapranov:Higher, GKT1, BOORS}), but we present them here so that we may return to them later when considering Frobenius and commutative structures.

In Section \ref{sec:nolift}, we briefly discuss the relationship between pseudomonoids in $\Span$ and monoids in the $1$-category $\Span_1$, which were studied in \cite{ContrerasMehtaSpan}. In particular, in Example \ref{ex:nolift}, we describe an infinite family of monoids in $\Span_1$ that do not admit lifts to pseudomonoid structures.

\subsection{\texorpdfstring{$2$}{2}-Segal sets}\label{sec:2Segal}

Let $X_\bullet$ be a simplicial set with face maps $d_i^n: X_n \to X_{n-1}$ and degeneracy maps $s_i^n: X_n \to X_{n+1}$. As is usual in the literature, we will sometimes drop the upper index when it is clear from context. 

For $n\geq 2$, any triangulation $\mathcal{T}$ of the regular $(n+1)$-gon (with vertices labeled from $0$ to $n$ clockwise) induces a map 
\begin{equation}\label{eqn:triangulation}
    \hat{\mathcal{T}} : X_n \to X_2 \times_{X_1} \cdots \times_{X_1} X_2,
\end{equation}
defined as follows. Each triangle in $\mathcal{T}$ is specified by its three vertices $v_0 < v_1 < v_2$, inducing a map $[2] = \{0,1,2\} \to [n] = \{0,\dots,n\}$ and thus a map $X_n \to X_2$, forming a component of $\hat{\mathcal{T}}$. The compatibility conditions between the components on the right of \eqref{eqn:triangulation} arise from the edges shared by triangles.

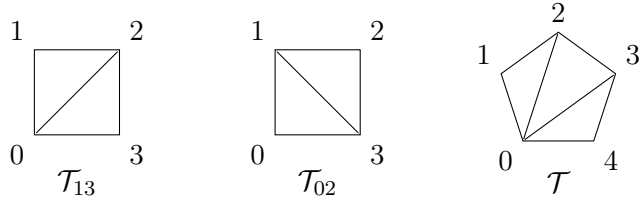
\begin{figure}[th]
\begin{center}
\begin{tikzpicture}[scale=0.8]
    \begin{scope}
        \draw (-135:1) node[vertex] (0) {} node[anchor=north east] {0}
        -- (-45:1) node[vertex] (3) {} node[anchor=north west] {3}
        -- (45:1) node[vertex] (2) {} node[anchor=south west] {2}
        -- (135:1) node[vertex] (1) {} node[anchor=south east] {1}
        -- cycle;
        \draw (0) -- (2);
        \node at (0,-1.5) {$\mathcal{T}_{13}$};
    \end{scope}
    \begin{scope}[shift={(4,0)}]
        \draw (-135:1) node[vertex] (0) {} node[anchor=north east] {0}
        -- (-45:1) node[vertex] (3) {} node[anchor=north west] {3}
        -- (45:1) node[vertex] (2) {} node[anchor=south west] {2}
        -- (135:1) node[vertex] (1) {} node[anchor=south east] {1}
        -- cycle;
        \draw (3) -- (1);
        \node at (0,-1.5) {$\mathcal{T}_{02}$};
    \end{scope}
    \begin{scope}[shift={(8,0)}]
        \draw (-126:1) node[vertex] (0) {} node[anchor=north east] {0}
        -- (-54: 1) node[vertex] (4) {} node[anchor=north west] {4}
        -- (18:1) node[vertex] (3) {} node[anchor=south west] {3}
        -- (90:1) node[vertex] (2) {} node[anchor=south] {2}
        -- (162:1) node[vertex] (1) {} node[anchor=south east] {1}
        -- cycle;
        \draw (0) -- (2);
        \draw (0) -- (3);
        \node at (0,-1.5) {$\mathcal{T}$};
    \end{scope}
\end{tikzpicture}
\end{center}
    \caption{The triangulations of the square and a triangulation of the pentagon.}
    \label{fig:square}
\end{figure}

The first nontrivial cases arise from the two triangulations of the square in Figure \ref{fig:square}. The associated maps $\hat{\mathcal{T}}_{13}$ and $\hat{\mathcal{T}}_{02}$ can be described in terms of the face maps as follows.
\begin{equation}
\label{eqn:taco}
\begin{split}
    \hat{\mathcal{T}}_{13}: X_3 &\to X_2 \bitimes{d_1}{d_2} X_2, \qquad   \hat{\mathcal{T}}_{02}: X_3 \to X_2 \bitimes{d_0}{d_1} X_2, \\
    \psi &\mapsto (d_3 \psi, d_1 \psi),  \qquad \qquad \quad \; \psi \mapsto (d_2 \psi, d_0 \psi).
\end{split}
\end{equation}
There are five triangulations of the pentagon, and one of them is shown in Figure \ref{fig:square}. The associated map $\hat{\mathcal{T}}$ is given by
\begin{equation}
\label{eqn:2segalx4}
    \begin{split}
    \hat{\mathcal{T}}: X_4 &\to X_2 \bitimes{d_1}{d_2} X_2 \bitimes{d_1}{d_2} X_2, \\
    \psi &\mapsto (d_{34} \psi, d_{14} \psi, d_{12} \psi),
    \end{split}
\end{equation}
where $d_{ij}: X_4 \to X_2$ arises from the monotonic injection $[2] \to [4]$ that skips $i$ and $j$. If we order the indices such that $i < j$, then we may explicitly write $d_{ij}$ as the composition of face maps $d_{ij} = d_i d_j$.

\begin{defn}\label{def:2segal}
    A simplicial set $X_\bullet$ is \emph{$2$-Segal} if, for every triangulation $\mathcal{T}$, the map $\hat{\mathcal{T}}$ is an isomorphism.
\end{defn}

The following example helps to illustrate the relationship between the $2$-Segal conditions and associativity. This relationship is described in more detail in Section \ref{sec:pseudospan}.
\begin{example}
    Let $G$ be a group, and let $X_\bullet$ be the nerve of $G$. Then the maps associated to the triangulations of the square are given by
    \begin{align*}
        \hat{\mathcal{T}}_{13} &: (g_1,g_2,g_3) \mapsto ((g_1 g_2, g_3), (g_1,g_2)), & \hat{\mathcal{T}}_{02} &: (g_1,g_2,g_3) \mapsto ((g_2, g_3), (g_1,g_2 g_3)).
    \end{align*}
    We can think of these two maps as corresponding to the two different bracketings for the triple product $g_1 g_2 g_3$; specifically, $\hat{\mathcal{T}}_{13}$ corresponds to $(g_1 g_2) g_3$, and $\hat{\mathcal{T}}_{02}$ corresponds to $g_1(g_2 g_3)$. The fact that $\hat{\mathcal{T}}_{13}$ and $\hat{\mathcal{T}}_{02}$ are isomorphisms means, roughly, that $X_3$ plays the role of an associator, connecting the two bracketings. 

    Similarly, the triangulations of the pentagon correspond to bracketings of the $4$-fold product $g_1 g_2 g_3 g_4$. For example, the triangulation of the pentagon in Figure \ref{fig:square} corresponds to $((g_1 g_2) g_3) g_4$. We note that the relationship between triangulations of the $(n+1)$-gon and bracketings of $n$-fold products is well-known. See, for instance \cite{LEE1989551} and \cite{HaimanAssoc}.
\end{example}

It is known \cite{FGKW:unital} that $2$-Segal sets satisfy the \emph{unitality conditions} that
\begin{equation}\label{eqn:unitality}
\begin{tikzcd}
    X_1 \arrow[r, "d_0"] \arrow[d, "s_1"] & X_0 \arrow[d, "s_0"] \\
    X_2 \arrow[r, "d_0"] & X_1
\end{tikzcd}
\hspace{1.5in}
\begin{tikzcd}
    X_1 \arrow[r, "d_1"] \arrow[d, "s_0"] & X_0 \arrow[d, "s_0"] \\
    X_2 \arrow[r, "d_2"] & X_1
\end{tikzcd}
\end{equation}
be pullback diagrams. There are also higher unitality conditions that are automatically satisfied (see \cite{BOORS, Dyckerhoff-Kapranov:Higher}). In particular, we will later use the fact that
\begin{equation}\label{eqn:higherunitality}
\begin{tikzcd}
    X_2 \arrow[r, "d_{02}"] \arrow[d, "s_1"] & X_0 \arrow[d, "s_0"] \\
    X_3 \arrow[r, "d_{03}"] & X_1
\end{tikzcd}    
\end{equation}
is a pullback diagram.

\subsection{More on \texorpdfstring{$2$}{2}-Segal sets: face, degeneracy, and edge maps}\label{sec:faceanddegen}

A nice feature of $2$-Segal sets is that high-dimensional simplices can be understood in terms of polygons, which can be drawn in two dimensions. This point of view leads to a graphical calculus for face and degeneracy maps.

Let $X_\bullet$ be a $2$-Segal set. To visualize the face map $d_i: X_n \to X_{n+1}$, choose a triangulation $\mathcal{T}$ of the $(n+1)$-gon that includes the triangle with vertex $i$ and its two adjacent vertices. Any element $\psi \in X_n$ can then be identified with its image under the corresponding isomorphism $\hat{\mathcal{T}}$ in \eqref{eqn:triangulation}. If we delete the component of $\psi$ corresponding to the triangle at vertex $i$, the remaining components form a triangulated $n$-gon whose vertices can be relabeled, giving us $d_i \psi$.

For example, the triangulation of the pentagon in Figure \ref{fig:square} gives an identification of any $\psi \in X_4$ with a triplet $(\xi,\xi',\xi'') \in X_2 \times_{X_1} X_2 \times_{X_1} X_2$, as in \eqref{eqn:2segalx4}. To get $d_1^4 \psi$, we delete $\xi$, which is the $2$-simplex corresponding to the triangle at vertex $1$. The remaining components $(\xi',\xi'')$ correspond, under the map $\hat{\mathcal{T}}_{13}$ in \eqref{eqn:taco}, to $d_1 \psi$. This process is illustrated in Figure \ref{fig:facemaps}.

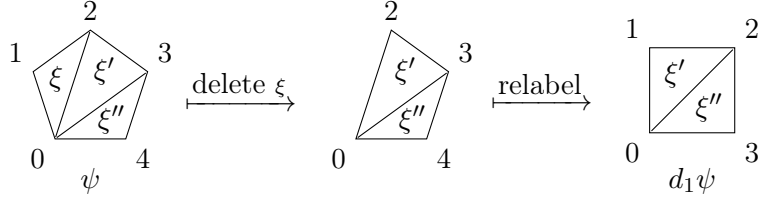
\begin{figure}[th]
\begin{center}
\begin{tikzpicture}[scale=0.8]
    \begin{scope}
        \draw (-126:1) node[vertex] (0) {} node[anchor=north east] {0}
        -- (-54: 1) node[vertex] (4) {} node[anchor=north west] {4}
        -- (18:1) node[vertex] (3) {} node[anchor=south west] {3}
        -- (90:1) node[vertex] (2) {} node[anchor=south] {2}
        -- (162:1) node[vertex] (1) {} node[anchor=south east] {1}
        -- cycle;
        \draw (0) -- (2);
        \draw (0) -- (3);
        \node at (162:.6) {$\xi$};
        \node at (54:.4) {$\xi'$};
        \node at (-54:.6) {$\xi''$};
        \node at (0,-1.5) {$\psi$};
    \end{scope}
    \node at (2.5,0) {$\xmapsto{\mbox{delete } \xi}$};
    \begin{scope}[shift={(5,0)}]
        \draw (-126:1) node[vertex] (0) {} node[anchor=north east] {0}
        -- (-54: 1) node[vertex] (4) {} node[anchor=north west] {4}
        -- (18:1) node[vertex] (3) {} node[anchor=south west] {3}
        -- (90:1) node[vertex] (2) {} node[anchor=south] {2}
        -- cycle;
        \draw (0) -- (3);
        \node at (54:.4) {$\xi'$};
        \node at (-54:.6) {$\xi''$};
    \end{scope}
    \node at (7.5,0) {$\xmapsto{\mbox{relabel}}$};
    \begin{scope}[shift={(10,0)}]
        \draw (-135:1) node[vertex] (0) {} node[anchor=north east] {0}
        -- (-45:1) node[vertex] (3) {} node[anchor=north west] {3}
        -- (45:1) node[vertex] (2) {} node[anchor=south west] {2}
        -- (135:1) node[vertex] (1) {} node[anchor=south east] {1}
        -- cycle;
        \draw (0) -- (2);
        \node at (135:.4) {$\xi'$};
        \node at (-45:.4) {$\xi''$};
        \node at (0,-1.5) {$d_1 \psi$};
    \end{scope}
\end{tikzpicture}
\end{center}
    \caption{Graphical calculus for face maps.}
    \label{fig:facemaps}
\end{figure}

Any $0 \leq i < j \leq n$ induces a map $[1] \to [n]$ and thus a map $X_n \to X_1$, which can be visualized as picking out the edge from $i$ to $j$ in the $(n+1)$-gon. Such an edge will only appear explicitly in the graphical calculus if the edge is part of the chosen triangulation. However, the exterior edges of the polygon are special, because they appear in \emph{every} triangulation. It will be useful to introduce notation for the maps associated to the exterior edges.

For $1 \leq i \leq n$, let $e_i^n: X_n \to X_1$ denote the map that picks out the edge from $i-1$ to $i$. The map that picks out the remaining edge, from $0$ to $n$, is denoted as $e_\out^n$; the reason for this notation will become clear in Section \ref{sec:pseudospan}.

To describe degeneracy maps, we first note that any $x \in X_1$ induces two degenerate $2$-simplices, $s_0 x$ and $s_1 x$. Each degenerate $2$-simplex has one degenerate edge, and the other two edges are $x$. 

Given a higher-dimensional simplex $\psi \in X_n$, $n \geq 2$, viewed as an $(n+1)$-gon, we can obtain $s_i \psi$ by attaching a degenerate $2$-simplex so that the resulting $(n+2)$-gon has degenerate edge $e_{i+1}$. For example, the images of $\psi \in X_2$ under degeneracy maps are shown in Figure \ref{fig:degeneracy}. Each has two presentations that are equivalent, in the sense that they relate to each other under the isomorphisms \eqref{eqn:taco} and thus represent the same element of $X_3$.

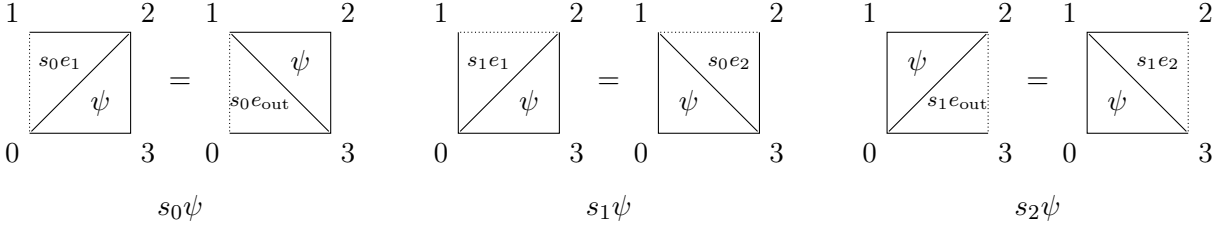
\begin{figure}[th]
\begin{center}
\begin{tikzpicture}[scale=0.95]

    \begin{scope}[shift={(0,0)}]
    \begin{scope}
        \draw (-135:1) node[vertex] (0) {} node[anchor=north east] {0}
        -- (-45:1) node[vertex] (3) {} node[anchor=north west] {3} 
        -- (45:1) node[vertex] (2) {} node[anchor=south west] {2}
        -- (135:1) node[vertex] (1) {} node[anchor=south east] {1};
        \draw[densely dotted] (0) -- (1);
        \draw (0) -- (2);
        \node at (135:.4) {${\scriptstyle s_0 e_1}$};
        \node at (-45:.4) {$\psi$};
    \end{scope}
    \equals{1.4}{0}
    \node at (1.4,-1.75) {$s_0 \psi$};    
    \begin{scope}[shift={(2.8,0)}]
        \draw (-135:1) node[vertex] (0) {} node[anchor=north east] {0}
        -- (-45:1) node[vertex] (3) {} node[anchor=north west] {3} 
        -- (45:1) node[vertex] (2) {} node[anchor=south west] {2}
        -- (135:1) node[vertex] (1) {} node[anchor=south east] {1};
        \draw[densely dotted] (0) -- (1);
        \draw (1) -- (3);
        \node at (45:.4) {$\psi$};
        \node at (-135:.4) {${\scriptstyle s_0 e_\out}$};
    \end{scope}
    \end{scope}
    
    \begin{scope}[shift={(6,0)}]
    \begin{scope}
        \draw (135:1) node[vertex] (1) {} node[anchor=south east] {1}
        -- (-135:1) node[vertex] (0) {} node[anchor=north east] {0}
        -- (-45:1) node[vertex] (3) {} node[anchor=north west] {3}
        -- (45:1) node[vertex] (2) {} node[anchor=south west] {2};
        \draw[densely dotted] (1) -- (2);
        \draw (0) -- (2);
        \node at (135:.4) {${\scriptstyle s_1 e_1}$};
        \node at (-45:.4) {$\psi$};
    \end{scope} 
    \equals{1.4}{0}
    \node at (1.4,-1.75) {$s_1 \psi$};    
    \begin{scope}[shift={(2.8,0)}]
        \draw (135:1) node[vertex] (1) {} node[anchor=south east] {1}
        -- (-135:1) node[vertex] (0) {} node[anchor=north east] {0}
        -- (-45:1) node[vertex] (3) {} node[anchor=north west] {3}
        -- (45:1) node[vertex] (2) {} node[anchor=south west] {2};
        \draw[densely dotted] (1) -- (2);
        \draw (1) -- (3);
        \node at (45:.4) {${\scriptstyle s_0 e_2}$};
        \node at (-135:.4) {$\psi$};
    \end{scope}
    \end{scope}
    \begin{scope}[shift={(12,0)}]
    \begin{scope}
        \draw (45:1) node[vertex] (2) {} node[anchor=south west] {2}
        -- (135:1) node[vertex] (1) {} node[anchor=south east] {1}
        -- (-135:1) node[vertex] (0) {} node[anchor=north east] {0}
        -- (-45:1) node[vertex] (3) {} node[anchor=north west] {3};
        \draw[densely dotted] (2) -- (3);
        \draw (0) -- (2);
        \node at (135:.4) {$\psi$};
        \node at (-45:.4) {${\scriptstyle s_1 e_\out}$};
    \end{scope} 
    \equals{1.4}{0}
    \node at (1.4,-1.75) {$s_2 \psi$};    
    \begin{scope}[shift={(2.8,0)}]
        \draw (45:1) node[vertex] (2) {} node[anchor=south west] {2}
        -- (135:1) node[vertex] (1) {} node[anchor=south east] {1}
        -- (-135:1) node[vertex] (0) {} node[anchor=north east] {0}
        -- (-45:1) node[vertex] (3) {} node[anchor=north west] {3};
        \draw[densely dotted] (2) -- (3);
        \draw (1) -- (3);
        \node at (45:.4) {${\scriptstyle s_1 e_2}$};
        \node at (-135:.4) {$\psi$};
    \end{scope}
    \end{scope}
    
\end{tikzpicture}
\end{center}
    \caption{Graphical calculus for degeneracy maps. Dotted edges are degenerate. Here, $s_i e_j$ is shorthand for $s_i e_j \psi$.}
    \label{fig:degeneracy}
\end{figure}

\subsection{Pseudomonoids in \texorpdfstring{$\Span$}{Span}}
\label{sec:pseudospan}

Instead of only considering triangulations of the regular $(n+1)$-gon, one can more generally consider subdivisions. More precisely, if $X_\bullet$ is a simplicial set, then any subdivision $\mathcal{T}$ of the $(n+1)$-gon into $(k_i+1)$-gons induces a map $\widehat{\mathcal{T}}$ from $X_n$ to an iterated pullback of the $X_{k_i}$ over copies of $X_1$. For example, in the case of the subdivision in Figure \ref{subfig:hexplain}, the associated map is
\[
\begin{tikzcd}
	X_5 \arrow[r] & X_2\times_{X_1}X_4.
\end{tikzcd}
\]
The $2$-Segal maps \eqref{eqn:triangulation} form the special case where the subdivision is a triangulation. The following result, which characterizes the 2-Segal conditions in terms of subdivisions, appears in \cite{Dyckerhoff-Kapranov:Higher}.

\begin{prop}[\cite{Dyckerhoff-Kapranov:Higher}, Proposition 2.3.2]\label{prop:subdivision}
	A simplicial set $X_\bullet$ satisfies the 2-Segal conditions if and only if, for every subdivision $\mathcal{T}$ of $P_{n+1}$ into polygons, the associated map $\widehat{\mathcal{T}}$ is an isomorphism. 
\end{prop}

\begin{figure}[th]
\begin{center}

\begin{subfigure}{0.225\textwidth}
\begin{center}
\begin{tikzpicture}[scale=0.75]
    \begin{scope}
        \draw (240:1) node[vertex] (0) {} node[anchor=north east] {0}
        -- (300:1) node[vertex] (5) {} node[anchor=north west] {5}
        -- (0:1) node[vertex] (4) {} node[anchor=west] {4}
        -- (60:1) node[vertex] (3) {} node[anchor=south west] {3}
        -- (120:1) node[vertex] (2) {} node[anchor=south east] {2}
        -- (180:1) node[vertex] (1) {} node[anchor=east] {1}        
        -- cycle;
        \draw (0) -- (2);
    \end{scope}
    \node at (0,-2.3) {};
\end{tikzpicture}
\caption{}
\label{subfig:hexplain}
\end{center}
\end{subfigure}
\begin{subfigure}{0.225\textwidth}
\begin{center}
\begin{tikzpicture}[scale=0.75]
    \begin{scope}
        \draw (240:1) node[vertex] (0) {} node[anchor=north east] {0}
        -- (300:1) node[vertex] (5) {} node[anchor=north west] {5}
        -- (0:1) node[vertex] (4) {} node[anchor=west] {4}
        -- (60:1) node[vertex] (3) {} node[anchor=south west] {3}
        -- (120:1) node[vertex] (2) {} node[anchor=south east] {2}
        -- (180:1) node[vertex] (1) {} node[anchor=east] {1}        
        -- cycle;
        \draw (0) -- (2);
    \end{scope}
    \begin{scope}[color=MidnightBlue]
        \node[vertex,draw,circle,fill] (A) at (0:.1) {};
        \node[vertex,draw,circle,fill] (B) at (180:.7) {};
        \node[vertex] (1) at (210:1.5) {};
        \node[vertex] (out) at (270:1.5) {};
        \node[vertex] (5) at (330:1.5) {};
        \node[vertex] (4) at (30:1.5) {};
        \node[vertex] (3) at (90:1.5) {};
        \node[vertex] (2) at (150:1.5) {};

        \draw (A) -- (B);
        \draw (A) -- (out) node[anchor=north] {out};
        \draw (A) -- (3) node[anchor=south] {3};
        \draw (A) -- (4) node[anchor=south west] {4};
        \draw (A) -- (5) node[anchor=north west] {5};
        \draw (B) -- (1) node[anchor=north east] {1};
        \draw (B) -- (2) node[anchor=south east] {2};
    \end{scope}
    \node at (0,-2.35) {};
\end{tikzpicture}
\caption{}
\label{subfig:hex}
\end{center}
\end{subfigure}
\begin{subfigure}{0.225\textwidth}
\begin{center}
\begin{tikzpicture}
    \begin{scope}[color=MidnightBlue, scale=1.1]
        \node[vertex,draw,circle,fill] (A) at (1,-.33) {};
        \node[vertex,draw,circle,fill] (B) at (.25,.33) {};
        \node[vertex] (out) at (1,-1) {};
        \node[vertex] (1) at (0,1) {};
        \node[vertex] (2) at (.5,1) {};
        \node[vertex] (3) at (1,1) {};
        \node[vertex] (4) at (1.5,1) {};
        \node[vertex] (5) at (2,1) {};

        \draw (A) -- (B);
        \draw (A) -- (out) node[anchor=north] {out};
        \draw (A) -- (3) node[anchor=south] {3};
        \draw (A) -- (4) node[anchor=south] {4};
        \draw (A) -- (5) node[anchor=south] {5};
        \draw (B) -- (1) node[anchor=south] {1};
        \draw (B) -- (2) node[anchor=south] {2};
    \end{scope}
\end{tikzpicture}
\end{center}
\caption{}
\label{subfig:tree}
\end{subfigure}
\begin{subfigure}{0.225\textwidth}
\begin{center}
\begin{tikzcd}
    (X_1)^5 \\
    X_2 \times_{X_1} X_4 \arrow{u} \arrow{d} \\
    X_1
\end{tikzcd}
\end{center}
\caption{}
\label{subfig:span}
\end{subfigure}
\end{center}
    \caption{(a) A subdivision of the hexagon. (b) The subdivided hexagon with the dual graph superimposed. (c) The dual graph redrawn as a labeled rooted tree. (d) The associated span.}
    \label{fig:subhex}
\end{figure}
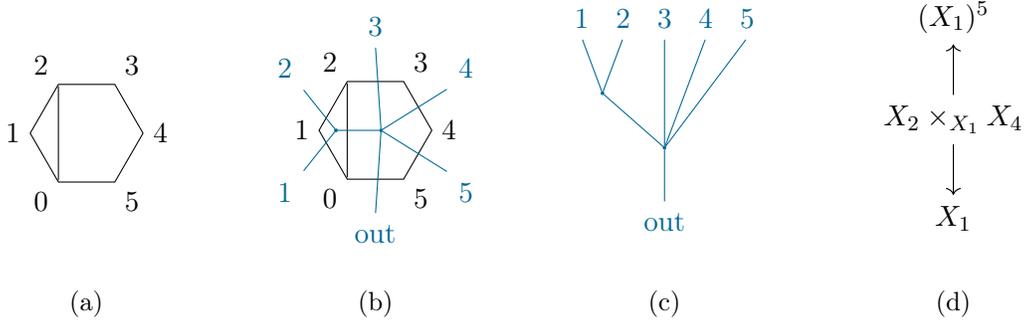

The relationship between the $2$-Segal conditions and associativity is elucidated by considering dual graphs. This perspective appears in \cite{dyckerhoff-kapranov:crossed}.
The dual graph of a subdivided $(n+1)$-gon is a labeled planar rooted tree, where the root is associated to the edge $e_\out$, and the $i$th leaf is associated to the edge $e_i$. For example, Figure \ref{subfig:hex} shows the dual graph superimposed on a subdivided hexagon, and Figure \ref{subfig:tree} shows the same graph drawn as a labeled rooted tree.

Let $X_\bullet$ be a (for now, not necessarily $2$-Segal) simplicial set. Then we can produce a span from a labeled rooted tree by assigning the span
\begin{equation}\label{eqn:xnspan}
   \begin{tikzcd}
   	(X_1)^n  &[2em] X_n\arrow[l, swap, "{(e_1, \dots, e_n)}"]\arrow[r, "e_{\out}"] & X_1 
   \end{tikzcd}
\end{equation} 
to each $(n+1)$-valent vertex. In particular, a trivalent vertex is sent to the span
\begin{equation}\label{eqn:mu}
\begin{tikzcd}
	X_1 \times X_1 &[2em] \arrow[l,"{(d_2,d_0)}"'] X_2 \arrow[r,"d_1"] & X_1,
\end{tikzcd}
\end{equation}
which defines a multiplication morphism $\mu$ in $\Span$. In this way, we can view a labeled planar rooted tree with $n$ leaves as being the string diagram for a span from $(X_1)^n$ to $X_1$.


The span associated to the unsubdivided $(n+1)$-gon is \eqref{eqn:xnspan}. A key observation is that the map $\widehat{\mathcal{T}}$ in Proposition \ref{prop:subdivision} is a morphism of spans from \eqref{eqn:xnspan} to the span associated to the subdivided $(n+1)$-gon.  Applying this map component-wise allows us to associate a map of spans to any refinement of a subdivision.

The upshot of all this is that any simplicial set $X_\bullet$ induces, for all $n \geq 2$, a ``$2$-Segal functor'' from the poset category of subdivided $(n+1)$-gons to the category $\Hom_{\Span}((X_1)^n, X_1)$. Then we can reinterpret Proposition \ref{prop:subdivision} as saying that $X_\bullet$ is $2$-Segal if and only if the image of this functor consists of isomorphisms of spans.

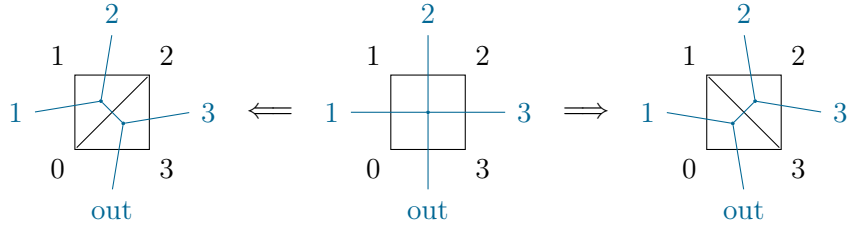
\begin{figure}[th]
\begin{center}
\begin{tikzpicture}[scale=0.7]
    \begin{scope}
        \draw (-135:1) node[vertex] (0) {} node[anchor=north east] {0}
        -- (-45:1) node[vertex] (3) {} node[anchor=north west] {3}
        -- (45:1) node[vertex] (2) {} node[anchor=south west] {2}
        -- (135:1) node[vertex] (1) {} node[anchor=south east] {1}
        -- cycle;
        \draw (0) -- (2);
    \end{scope}
    \begin{scope}[color=MidnightBlue]
        \node[vertex,draw,circle,fill] (B) at (135:.3) {};
        \node[vertex,draw,circle,fill] (A) at (315:.3) {};
        \node[vertex] (1) at (180:1.5) {};
        \node[vertex] (out) at (270:1.5) {};
        \node[vertex] (3) at (0:1.5) {};
        \node[vertex] (2) at (90:1.5) {};

        \draw (A) -- (B);
        \draw (A) -- (out) node[anchor=north] {out};
        \draw (A) -- (3) node[anchor=west] {3};
        \draw (B) -- (1) node[anchor=east] {1};
        \draw (B) -- (2) node[anchor=south] {2};
    \end{scope}
    \node at (3,0) {$\Longleftarrow$};
    \begin{scope}[shift={(6,0)}]
        \draw (-135:1) node[vertex] (0) {} node[anchor=north east] {0}
        -- (-45:1) node[vertex] (3) {} node[anchor=north west] {3}
        -- (45:1) node[vertex] (2) {} node[anchor=south west] {2}
        -- (135:1) node[vertex] (1) {} node[anchor=south east] {1}
        -- cycle;
    \end{scope}
        \begin{scope}[color=MidnightBlue, shift={(6,0)}]
        \node[vertex,draw,circle,fill] (A) at (0:0) {};
        \node[vertex] (1) at (180:1.5) {};
        \node[vertex] (out) at (270:1.5) {};
        \node[vertex] (3) at (0:1.5) {};
        \node[vertex] (2) at (90:1.5) {};

        \draw (A) -- (out) node[anchor=north] {out};
        \draw (A) -- (1) node[anchor=east] {1};
        \draw (A) -- (3) node[anchor=west] {3};
        \draw (A) -- (2) node[anchor=south] {2};
    \end{scope}
    \node at (9,0) {$\Longrightarrow$};
    \begin{scope}[shift={(12,0)}]
        \draw (-135:1) node[vertex] (0) {} node[anchor=north east] {0}
        -- (-45:1) node[vertex] (3) {} node[anchor=north west] {3}
        -- (45:1) node[vertex] (2) {} node[anchor=south west] {2}
        -- (135:1) node[vertex] (1) {} node[anchor=south east] {1}
        -- cycle;
        \draw (3) -- (1);
    \end{scope}
        \begin{scope}[color=MidnightBlue, shift={(12,0)}]
        \node[vertex,draw,circle,fill] (A) at (225:.3) {};
        \node[vertex,draw,circle,fill] (B) at (45:.3) {};
        \node[vertex] (1) at (180:1.5) {};
        \node[vertex] (out) at (270:1.5) {};
        \node[vertex] (3) at (0:1.5) {};
        \node[vertex] (2) at (90:1.5) {};

        \draw (A) -- (B);
        \draw (A) -- (out) node[anchor=north] {out};
        \draw (A) -- (1) node[anchor=east] {1};
        \draw (B) -- (3) node[anchor=west] {3};
        \draw (B) -- (2) node[anchor=south] {2};
    \end{scope}
\end{tikzpicture}
\end{center}
    \caption{The poset category of subdivided squares with dual graphs superimposed. This poset can be viewed as a cellular model of the associahedron $K_3$.}
    \label{fig:k3}
\end{figure}

We can now use these $2$-Segal functors to show that, if $X_\bullet$ is a $2$-Segal set, then the span \eqref{eqn:mu} is the multiplication for a pseudomonoid in $\Span$. To do this, we need to define an associator $a$ that satisfies the pentagon equation \ref{eqn:pentagon}, a unit morphism $\eta$, and unitors $r, \ell$ that satisfy the triangle equation \ref{eqn:triangle}.

For $n=3$, the poset of subdivided squares is a cellular model of the associahedron $K_3$, which relates the two different ways of multiplying three inputs with a binary operation; see Figure \ref{fig:k3}.  In this case, the image of the $2$-Segal functor is the following diagram in $\Hom_{\Span}((X_1)^3,X_1)$, consisting of the $2$-Segal maps in \eqref{eqn:taco}:
\stringdiagram
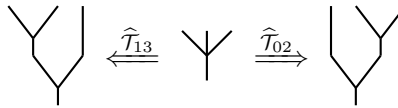
{
\begin{scope} 
\begin{scope}
\multiplication{0}{1}
\identity{2}{1}
\multiplication{1}{-1}
\end{scope}
\leftnattrans{4}{0}{\hat{\mathcal{T}}_{13}}
\begin{scope}[shift={(7,0)}]
\tripleprod{0}{0}    
\end{scope}
\nattrans{10}{0}{\hat{\mathcal{T}}_{02}}
\begin{scope}[shift={(12,0)}]
\identity{0}{1}
\multiplication{2}{1}
\multiplication{1}{-1}    
\end{scope}
\end{scope}
}
If $X_\bullet$ is $2$-Segal (so these maps are isomorphisms), then we can define an associator $a = \hat{\mathcal{T}}_{02} \circ \hat{\mathcal{T}}_{13}^{-1}$.

For $n=4$, the poset of subdivided pentagons is a cellular model of the associahedron $K_4$, which relates the five different ways of multiplying four inputs with a binary operation; see Figure \ref{fig:assoc4}. If $X_\bullet$ is $2$-Segal, then the commutativity of the associated diagram in $\Hom_{\Span}((X_1)^4,X_1)$ immediately implies that the associator $a$ satisfies the pentagon equation.

\begin{figure}[htb!]
\begin{center}
    \begin{tikzpicture}[scale=0.7]
            \begin{scope}
        \draw (-126:1) node[vertex] (0) {} 
        -- (-54: 1) node[vertex] (4) {} 
        -- (18:1) node[vertex] (3) {} 
        -- (90:1) node[vertex] (2) {} 
        -- (162:1) node[vertex] (1) {} 
        -- cycle;
    \end{scope}
    \begin{scope}[shift={(54:5)}]
        \draw (-126:1) node[vertex] (0) {} 
        -- (-54: 1) node[vertex] (4) {} 
        -- (18:1) node[vertex] (3) {} 
        -- (90:1) node[vertex] (2) {} 
        -- (162:1) node[vertex] (1) {} 
        -- cycle;
        \draw (0) -- (2);
    \end{scope}
        \begin{scope}[shift={(90:6)}]
        \draw (-126:1) node[vertex] (0) {} 
        -- (-54: 1) node[vertex] (4) {} 
        -- (18:1) node[vertex] (3) {} 
        -- (90:1) node[vertex] (2) {} 
        -- (162:1) node[vertex] (1) {} 
        -- cycle;
        \draw (0) -- (2);
        \draw (0) -- (3);
    \end{scope}
        \begin{scope}[shift={(126:5)}]
        \draw (-126:1) node[vertex] (0) {} 
        -- (-54: 1) node[vertex] (4) {} 
        -- (18:1) node[vertex] (3) {} 
        -- (90:1) node[vertex] (2) {} 
        -- (162:1) node[vertex] (1) {} 
        -- cycle;
        \draw (0) -- (3);
    \end{scope}
        \begin{scope}[shift={(162:6)}]
        \draw (-126:1) node[vertex] (0) {} 
        -- (-54: 1) node[vertex] (4) {} 
        -- (18:1) node[vertex] (3) {} 
        -- (90:1) node[vertex] (2) {} 
        -- (162:1) node[vertex] (1) {} 
        -- cycle;
        \draw (0) -- (3);
        \draw (1) -- (3);
    \end{scope}
        \begin{scope}[shift={(198:5)}]
        \draw (-126:1) node[vertex] (0) {} 
        -- (-54: 1) node[vertex] (4) {} 
        -- (18:1) node[vertex] (3) {} 
        -- (90:1) node[vertex] (2) {} 
        -- (162:1) node[vertex] (1) {} 
        -- cycle;
        \draw (1) -- (3);
    \end{scope}
         \begin{scope}[shift={(234:6)}]
        \draw (-126:1) node[vertex] (0) {} 
        -- (-54: 1) node[vertex] (4) {} 
        -- (18:1) node[vertex] (3) {} 
        -- (90:1) node[vertex] (2) {} 
        -- (162:1) node[vertex] (1) {} 
        -- cycle;
        \draw (1) -- (3);
        \draw (1) -- (4);
    \end{scope}
        \begin{scope}[shift={(270:5)}]
        \draw (-126:1) node[vertex] (0) {} 
        -- (-54: 1) node[vertex] (4) {} 
        -- (18:1) node[vertex] (3) {} 
        -- (90:1) node[vertex] (2) {} 
        -- (162:1) node[vertex] (1) {} 
        -- cycle;
        \draw (1) -- (4);
    \end{scope}
        \begin{scope}[shift={(306:6)}]
        \draw (-126:1) node[vertex] (0) {} 
        -- (-54: 1) node[vertex] (4) {} 
        -- (18:1) node[vertex] (3) {} 
        -- (90:1) node[vertex] (2) {} 
        -- (162:1) node[vertex] (1) {} 
        -- cycle;
        \draw (1) -- (4);
        \draw (2) -- (4);
    \end{scope}
        \begin{scope}[shift={(342:5)}]
        \draw (-126:1) node[vertex] (0) {} 
        -- (-54: 1) node[vertex] (4) {} 
        -- (18:1) node[vertex] (3) {} 
        -- (90:1) node[vertex] (2) {} 
        -- (162:1) node[vertex] (1) {} 
        -- cycle;
        \draw (2) -- (4);
    \end{scope}
        \begin{scope}[shift={(18:6)}]
        \draw (-126:1) node[vertex] (0) {} 
        -- (-54: 1) node[vertex] (4) {} 
        -- (18:1) node[vertex] (3) {} 
        -- (90:1) node[vertex] (2) {} 
        -- (162:1) node[vertex] (1) {} 
        -- cycle;
        \draw (2) -- (4);
        \draw (0) -- (2);
    \end{scope}
    \node[rotate=342] at (342:2.5) {$\Longrightarrow$};
    \node[rotate=54] at (54:2.5) {$\Longrightarrow$};
    \node[rotate=126] at (126:2.5) {$\Longrightarrow$};
    \node[rotate=198] at (198:2.5) {$\Longrightarrow$};
    \node[rotate=270] at (270:2.5) {$\Longrightarrow$};
    \node[rotate=72] at (0:5.5) {$\Longrightarrow$};    
    \node[rotate=144] at (72:5.5) {$\Longrightarrow$};
    \node[rotate=216] at (144:5.5) {$\Longrightarrow$};
    \node[rotate=288] at (216:5.5) {$\Longrightarrow$};
    \node[rotate=0] at (288:5.5) {$\Longrightarrow$};
    \node[rotate=324] at (36:5.5) {$\Longrightarrow$};    
    \node[rotate=36] at (108:5.5) {$\Longrightarrow$};
    \node[rotate=108] at (180:5.5) {$\Longrightarrow$};
    \node[rotate=180] at (252:5.5) {$\Longrightarrow$};
    \node[rotate=252] at (324:5.5) {$\Longrightarrow$};    
    \end{tikzpicture}
\end{center}

\caption{The poset of subdivided pentagons is a cellular model of the associahedron $K_4$.}
\label{fig:assoc4} 
\end{figure}
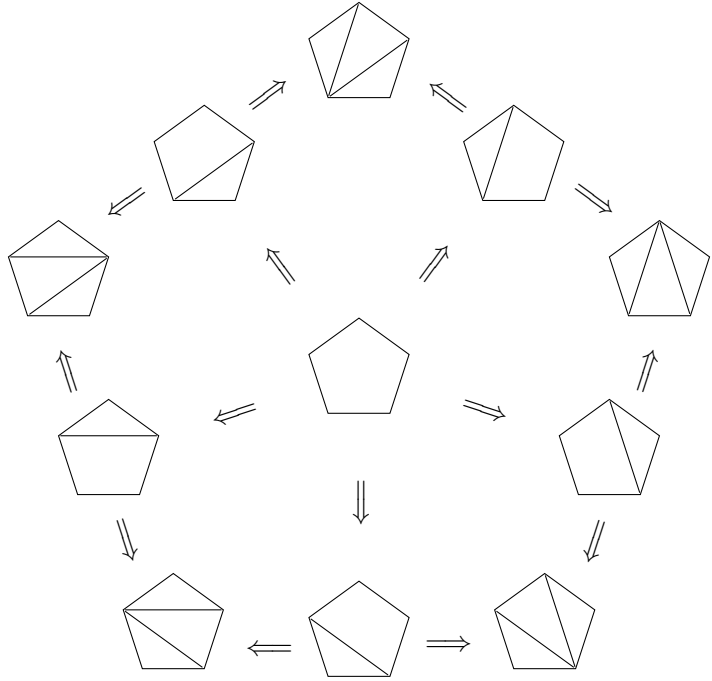

The unit morphism $\eta$ associated to $X_\bullet$ is defined by the span
\[ 
\begin{tikzcd}
	\pt & \arrow[l] X_0 \arrow[r,"s_0"] & X_1.
\end{tikzcd}
\]
The composition $\mu \circ (\eta \times \id_{X_1})$ can be canonically identified with the span
\[
\begin{tikzcd}[row sep=0em]
	X_1 & X_0 \bitimes{s_0}{d_2}X_2\arrow[r] \arrow[l] & X_1,\\
	d_0\xi & (u,\xi)\arrow[r,mapsto]\arrow[l,mapsto] & d_1\xi.
\end{tikzcd}
\]
The map $X_1 \xrightarrow{(d_1,s_0)} X_0 \bitimes{s_0}{d_2} X_2$ is a map of spans $\id_{X_1} \Rightarrow \mu \circ (\eta \times \id_{X_1})$. If $X_\bullet$ is $2$-Segal, then it follows from the unitality property \eqref{eqn:unitality} that this map is an isomorphism. The inverse map then gives the unitor $2$-morphism $\ell: \mu \circ (\eta \times \id_{X_1}) \Rightarrow \id_{X_1}$. The other unitor $2$-morphism $r$ is obtained in a similar way.

To verify the {triangle equation} \ref{eqn:triangle}, we consider the diagram
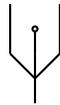
\stringdiagram{
\identity{-1}{1}
\unit{0}{1}
\identity{1}{1}
\tripleprod{0}{-1}
}
for which the associated span is canonically isomorphic to
\begin{equation}\label{eqn:s1}
\begin{tikzcd}[row sep=0em]
	(X_1)^2 & X_0 \bitimes{s_0}{e_2}X_3\arrow[r] \arrow[l] & X_1,\\
	(e_1\xi,e_3\xi) & (u,\xi)\arrow[r,mapsto]\arrow[l,mapsto] & e_\out \xi.
\end{tikzcd}
\end{equation}
The map $S: X_2 \xrightarrow{(d_{02},s_1)} X_0 \bitimes{s_0}{e_2} X_3$ is a map of spans from $\mu$ to \eqref{eqn:s1}. If $X_\bullet$ is $2$-Segal, then it follows from \eqref{eqn:higherunitality} that $S$ is an isomorphism. We then have the following diagram in $\Hom_{\Span}((X_1)^2,X_1)$:
\stringdiagram{
\begin{scope} 
\begin{scope}
\identity{-1}{2}
\unit{1}{2}
\identity{2}{2}
\multiplication{0}{0}
\identity{2}{0}
\multiplication{1}{-2}
\end{scope}
\leftnattrans{4}{0}{\hat{\mathcal{T}}_{13}}
\begin{scope}[shift={(7,0)}]
\identity{-1}{1}
\unit{0}{1}
\identity{1}{1}
\tripleprod{0}{-1}   
\end{scope}
\nattrans{10}{0}{\hat{\mathcal{T}}_{02}}
\begin{scope}[shift={(12,0)}]
\identity{-0}{2}
\unit{1}{2}
\identity{3}{2}
\identity{0}{0}
\multiplication{2}{0}
\multiplication{1}{-2}    
\end{scope}
\upnattrans{7.4}{-3.5}{S}
\begin{scope}[shift={(7,-6)}]
\multiplication{0}{0}    
\end{scope}
\end{scope}
}
Now we observe that the maps in the triangle equation can be expressed in terms of the maps in the above diagram:
\begin{align*}
    a &= \hat{\mathcal{T}}_{02} \hat{\mathcal{T}}_{13}^{-1},&
    r &= S^{-1} \hat{\mathcal{T}}_{13}^{-1},&
    \ell &= S^{-1} \hat{\mathcal{T}}_{02}^{-1}.
\end{align*}
The triangle equation $r = \ell \circ a$ follows.

We will only briefly sketch the other direction of the equivalence, as later sections primarily make use of the direction that we have described. Given a pseudomonoid in $\Span$, let us suggestively denote the underlying object by $Y_1$, the unit by 
\[
\begin{tikzcd}
	\pt & Y_0 \arrow[l]\arrow[r,"s_0"] & Y_1 
\end{tikzcd}
\]
and the multiplication by 
\[
\begin{tikzcd}
	Y_1\times Y_1 & Y_2 \arrow[r,"d_1"] \arrow[l,"{(d_2,d_0)}"'] & Y_1.
\end{tikzcd}
\]
By iterated composition of the multiplication with itself, we can obtain an $n$-fold multiplication 
\[
\begin{tikzcd}
	Y_1^{\times n} & Y_n \arrow[l] \arrow[r] & Y_1. 
\end{tikzcd}
\] 
We can then construct a simplicial set whose set of $n$-simplices is $Y_n$. The degeneracy maps are induced by composition with the unit, and the face maps are induced by restriction to subproducts, sometimes requiring the use of the associator. The $2$-Segal conditions hold as a result of associativity: for $n>2$, $Y_n$ is constructed as an iterated composite of 2-fold multiplications, and by associativity, the canonical 2-morphism (associator) between any two such composites must be an isomorphism.

\subsection{Examples}\label{sec:examples2segal}

In this section we review some examples of $2$-Segal sets. We will return to these examples later, in Sections \ref{subsec:sym_egs} and \ref{subsec:para_egs}, where we will consider when the corresponding pseudomonoids in $\Span$ admit Frobenius and/or commutative structures.

\begin{example}[Categories]\label{ex:categories}
If $\C$ is a small category, then the nerve $N\C$ of $\C$ is a $2$-Segal set; see \cite[Proposition 2.3.4]{Dyckerhoff-Kapranov:Higher} for a proof.
\end{example}

\begin{example}[Partial monoids]\label{ex:partialmonoids}
The notion of partial monoid and its nerve construction appeared in \cite{Segal:Conf}, and it was shown in \cite{BOORS} that the nerve of a partial monoid is $2$-Segal. We briefly review the construction here.

A \emph{partial monoid} is a set $M$ equipped with a partially-defined multiplication map $M_2 \to M$, $(x,x') \mapsto x \cdot x'$, for some subset $M_2 \subset M \times M$, satisfying the following conditions. In the following, all equations should be taken to mean that either both sides are undefined or both sides are defined and equal.
\begin{enumerate}
    \item (Associativity) For all $x,x',x'' \in M$, 
    \[ (x \cdot x') \cdot x'' = x \cdot (x' \cdot x''),\]
    \item (Unitality) There exists $1 \in M$ such that $1 \cdot x = x \cdot 1 = x$ for all $x \in M$.
\end{enumerate}

The nerve $N_\bullet M$ of a partial monoid $M$ has $N_0 M = \pt$ and, for $n \geq 1$, $N_n M$ consists of fully composable $n$-tuples of elements of $M$. The face and degeneracy maps are given by the same formulas as those for the nerve of a monoid. We refer to \cite{BOORS} for details as well as a proof that the $2$-Segal conditions hold.
\end{example}

\begin{example}[Twisted cyclic nerves]\label{ex:twistedcyclic}
Let $\C$ be a small category equipped with an endofunctor $F: \C \to \C$. The twisted cyclic nerve $N^F_\bullet \C$ is defined as follows. The set $N^F_n \C$ consists of composable $(n+1)$-tuples $(f_n, \dots, f_0)$ of morphisms in $\C$ such that the composition $F(f_0) f_n$ is defined. For $0<i\leq n$, the face map $d_i^n$ is given by
\[ d_i^n(f_n, \dots, f_0) = (f_n, \dots, f_i f_{i-1}, \dots, f_0),\]
and $d_0^n$ is given by
\[ d_0^n(f_n, \dots, f_0) = (F(f_0) f_n, f_{n-1}, \dots, f_1).\]
The degeneracy map $s_i^n$ inserts an identity morphism after $f_i$. It is shown in \cite[Theorem 3.2.3]{Dyckerhoff-Kapranov:Higher} that $N^F_\bullet \C$ is $2$-Segal.
\end{example}

\begin{example}[Twisted inertia groupoids]\label{ex:twistedinertia}
Let $G$ be a group equipped with an endomorphism $F: G \to G$. Then, as a special case of Example \ref{ex:twistedcyclic}, we obtain the $2$-Segal set $N^F_\bullet G$. This can be seen to be isomorphic to the nerve of a groupoid for which the set of objects is $G$, and where the set of morphisms is $G \times G$. Specifically, $(g,h) \in G \times G$ is a morphism from $h$ to $F(g) h g^{-1}$. In the case where $F = \id$, this construction produces the \emph{inertia groupoid} of $G$, which encodes the conjugation action of $G$ on itself. More generally, we could view it as a \emph{twisted inertia groupoid}, where the conjugation action is twisted by $F$.
\end{example}

\begin{example}[Buildings]\label{ex:buildings}
Let $X$ be a poset equipped with an order-preserving map $F: X \to X$. Then, as another special case of Example \ref{ex:twistedcyclic}, we obtain the $2$-Segal set $N^F_\bullet X$, where
\[ N_n^F X = \{ (x_0, \dots, x_n) \suchthat x_0 \leq \dots \leq x_n \leq F(x_0) \},\]
where the face map $d_i^n$ deletes $x_i$, and where the degeneracy map $s_i^n$ duplicates $x_i$. In this case, $N^F_\bullet X$ is called the \emph{building} of $X$. See \cite[Proposition 3.1.4]{Dyckerhoff-Kapranov:Higher} for a direct proof that buildings are $2$-Segal.
\end{example}

\begin{example}[Graph partitions]\label{ex:graphpartitions}
Let $G$ be a graph. In \cite[Example 2.3]{BOORS}, an associated $2$-Segal set $X(G)_\bullet$ is given via a modification of the construction of \cite[Example 1.1.5]{GKT_comb}. Here, the elements of $X(G)_n$ are of the form $(H;S_1, \dots, S_n)$, where $H$ is a subgraph of $G$ and $(S_1, \dots, S_n)$ is a partition of the set of vertices of $H$. We omit a detailed description for now, since we will see in Section \ref{subsec:sym_egs} that this example fits more naturally in the more structured framework of $\Gamma$-sets.
\end{example}

\subsection{Relation to monoids in the \texorpdfstring{$1$}{1}-category of spans}\label{sec:nolift}

Recall that $\Span_1$ denotes the homotopy $1$-category of spans of sets. In \cite{ContrerasMehtaSpan}, monoids in $\Span_1$ were similarly described in terms of simplicial sets, yielding conditions that are weaker than the $2$-Segal conditions. The simplicial sets considered there are $2$-truncated, meaning they only include data up to $X_2$. One can then form the \emph{taco spaces} $X_2 \bitimes{d_1}{d_2} X_2$ and $X_2 \bitimes{d_0}{d_1} X_2$ that correspond to the two triangulations of the square (see \eqref{eqn:taco}). Via the edge maps $e_1$, $e_2$, $e_3$, and $e_\out$, we can view the taco spaces as spans from $(X_1)^3$ to $X_1$.

In this $1$-categorical setting, the associativity condition is the existence of an associator, i.e.\ an isomorphism of spans $X_2 \bitimes{d_1}{d_2} X_2 \Rightarrow X_2 \bitimes{d_0}{d_1} X_2$. We stress that one only needs existence; the associator is not part of the data, and there is no coherence condition.

In general, pseudomonoids in a bicategory descend to monoids in the homotopy $1$-category, and thus a pseudomonoid in $\Span$ descends to a pseudomonoid in $\Span_1$. On the simplicial set side, we see this relationship in the fact that, if $X_\bullet$ is a $2$-Segal set, then we have a natural associator $a = \hat{\mathcal{T}}_{02} \circ \hat{\mathcal{T}}_{13}^{-1}$, which witnesses the $1$-categorical associativity condition. 

In the other direction, given a monoid in $\Span_1$, one could ask whether there exists a lift to a pseudomonoid in $\Span$. In terms of simplicial sets, this amounts to asking whether it is possible to choose an associator $X_2 \bitimes{d_1}{d_2} X_2 \Rightarrow X_2 \bitimes{d_0}{d_1} X_2$ such that the pentagon equation is satisfied.

It turns out that the existence of such a lift can be somewhat restrictive, as the following example shows. 

\begin{example}[Monoids that do not lift to pseudomonoids]\label{ex:nolift}
Consider the $2$-truncated simplicial set given by $X_0 = \{0\}$, $X_1 = \{0,1\}$, and $X_2 = \{(0,0),(1,0),(0,1)\} \sqcup A$, where $A$ is an arbitrary set. The face maps are given by
\begin{align*}
    d_0(k, \ell) &= \ell, & d_1(k,\ell) &= k+\ell, & d_2(k, \ell) &= k,
\end{align*}
for $(k,\ell) \in \{(0,0),(1,0),(0,1)\}$, and
\begin{align*}
    d_0(a) &= 1, & d_1(a) &= 0, & d_2(a) &= 1,
\end{align*}
for $a \in A$. The degeneracy maps are given by $s_0^0(0) = 0$ and
\begin{align*}
    s_0^1(k) &= (0,k), & s_1^1(k) &= (k,0),
\end{align*}
for $k \in X_1$. Intuitively, we can think of $X_\bullet$ as being similar to the ($2$-truncation of the) nerve of the group $\mathbb{Z}_2$, except that the $2$-simplices that represent the sum $1+1=0$ are labeled by elements of $A$. We observe that, when $|A| = 1$, $X_\bullet$ is the nerve of $\mathbb{Z}_2$, and when $A = \emptyset$, $X_\bullet$ is the nerve of a partial monoid.

We can partition the taco spaces as
\begin{align*}
    X_2 \bitimes{d_1}{d_2} X_2 &= \bigsqcup M_{ijk}, & X_2 \bitimes{d_0}{d_1} X_2 &= \bigsqcup M'_{ijk},
\end{align*}
where the indices $i,j,k \in X_1$ indicate the images under the edge maps $e_1,e_2,e_3$. Note that, in this example, $e_\out$ is uniquely determined by the other three edges. An associator is equivalent to a collection of isomorphisms $M_{ijk} \cong M'_{ijk}$.

In the cases $ijk = 000, 100, 010, 001$, the sets $M_{ijk}$ and $M'_{ijk}$ are both singletons, so there is a unique isomorphism between them. For example, 
\begin{align*}
M_{100} &= \{((1,0),(1,0))\}, & M'_{100} &= \{((1,0),(0,0))\}. 
\end{align*}
In the other cases, $M_{ijk}$ and $M'_{ijk}$ can both be identified with $A$. Specifically, 
\begin{align*}
M_{110} &= \{(a, (0,0)) \suchthat a \in A\}, & M'_{110} &= \{(a, (1,0)) \suchthat a \in A\}, \\
M_{101} &= \{((1,0), a) \suchthat a \in A\}, & M'_{101} &= \{(a,(0,1)) \suchthat a \in A\}, \\
M_{011} &= \{((0,1),a) \suchthat a \in A\}, & M'_{011} &= \{((0,0),a) \suchthat a \in A \}, \\
M_{111} &= \{(a,(0,1)) \suchthat a \in A\}, & M'_{111} &= \{((1,0),a) \suchthat a \in A\}.
\end{align*}
The existence of an associator is clear. This example is a special case of \cite[Example 3.8]{ContrerasMehtaSpan} and \cite[Example 2.8]{KenneyPare}.

In \cite{KenneyPare}, it is further claimed that, because the associator can be canonically chosen, it must satisfy the pentagon equation. However, it turns out that this is not true. The set $X_2 \bitimes {d_1}{d_2} X_2 \bitimes{d_1}{d_2} X_2$ corresponds to the triangulation of the pentagon in Figure \ref{fig:square} (see \eqref{eqn:2segalx4}). As with the triangulations of the square, we have a partition
\[ X_2 \bitimes {d_1}{d_2} X_2 \bitimes{d_1}{d_2} X_2 = \bigsqcup P_{ijk\ell},\]
where the indices $i,j,k,\ell \in X_1$ indicate the images under the edge maps $e_1,e_2,e_3,e_4$. It is then a straightforward check to see that
\[ P_{1111} = \{(a,(0,1),a') \suchthat a,a' \in A\}, \]
and that the automorphism of $P_{1111}$ obtained by using the canonical associator to go around the pentagon diagram is the map $(a,(0,1),a') \mapsto (a', (0,1),a)$, which is different from the identity map, except when $|A| = 0$ or $1$. More generally, for any choice of associator, the pentagon diagram gives a map of the form $(a, (0,1), a') \mapsto (\phi'(a'), (0,1), \phi(a))$, where $\phi,\phi'$ are automorphisms of $A$.

To summarize, in this example we have described an infinite family of monoids in $\Span_1$, parametrized by the set $A$, which only admit lifts to pseudomonoids in $\Span$ in the cases $|A| = 0,1$.
\end{example}

\section{Paracyclic structures and Frobenius pseudomonoids in \texorpdfstring{$\Span$}{Span}}
\label{sec:paracyclicfrobenius}

In this section, we consider Frobenius structures on pseudomonoids in $\Span$. We find that Frobenius structures correspond to \emph{paracyclic} structures on $2$-Segal sets.

In Section \ref{sec:paracycliccat}, we review the definition and basic properties of the paracyclic category $\Lambda_\infty$, which is closely related to the cyclic category of Connes. The paracyclic category has appeared in the study of cyclic homology \cites{getzler-jones, nistor} and crossed simplicial groups \cites{dyckerhoff-kapranov:crossed, fiedorowicz-loday}. 

A \emph{paracyclic set} is a functor $\Lambda_\infty^\op \to \Set$. There is a natural inclusion of the simplex category $\Delta$ into $\Lambda_\infty$, so one can think of a paracyclic set as a simplicial set with some additional structure. This additional structure is explicitly written down in Section \ref{sec:paracyclicset}. Then, in Section \ref{sec:frobeniusspan} we prove one of our main results, that a paracyclic structure on a $2$-Segal set is equivalent to a Frobenius structure on the corresponding pseudomonoid in $\Span$. We conclude the section with examples in Section \ref{subsec:para_egs}.

\subsection{The paracyclic category}\label{sec:paracycliccat}

The paracyclic category $\Lambda_\infty$ has the same objects $[n]$ as the simplex category $\Delta$. A morphism from $[m]$ to $[n]$ is defined to be an order-preserving map $f: \mathbb{Z} \to \mathbb{Z}$ such that 
\begin{equation}\label{eqn:paracyclic}
    f(i + k(m+1)) = f(i) + k(n+1)
\end{equation}
for all $k \in \mathbb{Z}$. This definition of $\Lambda_\infty$ may seem a bit mysterious, so it may help to give some intuition for how it arises. For each $n$, consider the covering map $\mathbb{Z} \to [n]$, $i \mapsto i \pmod{n+1}$. Associated to this covering map is the action of $\mathbb{Z}$ on itself by deck transformations $\varphi_k^n: \mathbb{Z} \to \mathbb{Z}$, $i \mapsto i + k(n+1)$. Condition \eqref{eqn:paracyclic} can then be interpreted as saying that $f$ is equivariant, i.e.\ $f \circ \varphi_k^m = \varphi_k^n \circ f$.

A consequence of the equivariance condition is that any morphism $f \in \Hom_{\Lambda_\infty}([m],[n])$ covers an underlying map $\tilde{f}: [m] \to [n]$. The fact that $f$ is order-preserving implies that $\tilde{f}$ preserves cyclic order, so $f \mapsto \tilde{f}$ defines a functor from $\Lambda_\infty$ to the cyclic category $\Lambda$.

Condition \eqref{eqn:paracyclic} implies that any $f \in \Hom_{\Lambda_\infty}([m],[n])$ is determined by its values on $\{0,\dots,m\}$. Conversely, any order-preserving map $f: \{0,\dots,m\} \to \mathbb{Z}$ such that $f(m) \leq f(0) + n + 1$ uniquely extends to a morphism $f \in \Hom_{\Lambda_\infty}([m],[n])$. As a result, we can identify the simplex category $\Delta$ with the subcategory of $\Lambda_\infty$ consisting of morphisms $f \in \Hom_{\Lambda_\infty}([m],[n])$ that map $\{0,\dots,m\}$ into $\{0,\dots,n\}$. In particular, we have the coface maps $\delta_i^n \in \Hom_{\Delta}([n-1],[n])$ and the codegeneracy maps $\sigma_i^n \in \Hom_{\Delta}([n+1],[n])$.

There are two special families of morphisms in $\Lambda_\infty$ that are not in $\Delta$. For each $[n]$, let $\sigma_{n+1}^n \in \Hom_{\Lambda_\infty}([n+1],[n])$ be given by $\sigma_{n+1}^n(i) = i$ for $i \in \{0,\dots,n+1\}$, and let $T^n \in \Hom_{\Lambda_\infty}([n], [n])$ be given by $T^n(i) = i+1$ for all $i$. A direct calculation shows that these morphisms are related to each other via the equation
\begin{equation}\label{eqn:sigmatau}
    \sigma_{n+1}^n \delta_0^{n+1} = T^n.
\end{equation} 
The notation is intended to suggest that $\sigma_{n+1}^n$ should be viewed as an extra codegeneracy map. The role of $T^n$ is that it is invertible and is a generator of $\Aut_{\Lambda_\infty}([n]) \cong \mathbb{Z}$.

An important fact about the paracyclic category is the following unique factorization property. The proof is usually omitted in the literature, but we include it here because it is constructive and useful for calculations.
\begin{prop}\label{prop:paracyclicfactor}
    Every morphism $f \in \Hom_{\Lambda_\infty}([m],[n])$ can be uniquely factored in the form $f = g \circ (T^m)^{-a}$, where $g \in \Hom_\Delta([m],[n])$ and $a \in \mathbb{Z}$.
\end{prop}
\begin{proof}
    Let $a$ be the minimum value such that $f(a) \geq 0$. Setting $g = f \circ (T^m)^a$, we see that $g(0) = f(a) \geq 0$ and $g(m) = f(m+a) = f(a-1 + (m+1)) = f(a-1) + n+1 \leq n$, so $g$ maps $\{0,\dots,m\}$ to $\{0,\dots,n\}$. For uniqueness, one can check that any smaller value of $a$ would have $g(0) < 0$, and any larger value of $a$ would have $g(m) > n$.
\end{proof}
It follows from Proposition \ref{prop:paracyclicfactor} that $\Lambda_\infty$ is generated by $T^n$ (and its inverse), the coface maps $\delta_i^n$, and the codegeneracy maps $\sigma_i^n$. One can also use the factorization of Proposition \ref{prop:paracyclicfactor} to obtain the following relations which, together with the cosimplicial relations, are sufficient to completely characterize $\Lambda_\infty$.
\begin{align} \label{eqn:paracycliccoface}
    T^n \delta_i^n &= \begin{cases}
        \delta_{i+1}^n T^{n-1}, & 0 \leq i < n, \\
        \delta_0^n, & i=n,
    \end{cases}\\ \label{eqn:paracycliccodegen}
    T^n \sigma_i^n &= \begin{cases}
        \sigma_{i+1}^n T^{n+1}, & 0 \leq i < n, \\
        \sigma_0^n (T^{n+1})^2, & i=n.
    \end{cases}
\end{align}
We note that \eqref{eqn:paracycliccoface}--\eqref{eqn:paracycliccodegen} differ from other sources (e.g.\ \cite{fiedorowicz-loday, getzler-jones}) due to a difference of convention in the indexing of the cosimplicial maps.

As a special case of Proposition \ref{prop:paracyclicfactor}, we have $\sigma_{n+1}^n = \sigma_0^n T^{n+1}$. From this, one can obtain the following relations:
\begin{align}\label{eqn:paracycliclift1}
    \sigma_{n+1}^n \delta_i^{n+1} &= \begin{cases}
        T^n, & i=0,\\
        \delta_i^n \sigma_n^{n-1}, & 0<i<n+1,\\
        \id, & i=n+1,
    \end{cases}\\ \label{eqn:paracycliclift2}
    \sigma_{n+1}^n \sigma_i^{n+1} &=
        \sigma_i^n \sigma_{n+2}^{n+1}, \;\;\;\;\;\; 0 \leq i \leq n+1.
\end{align}
The interpretation of $\sigma_{n+1}^n$ as an extra codegeneracy map is justified by the fact that these relations agree with the usual cosimplicial relations, with only one exceptional rule $\sigma_{n+1}^n \delta_0^{n+1} = T^n$.

\subsection{Paracyclic sets}\label{sec:paracyclicset}

A paracyclic set is defined to be a functor $\Lambda_\infty^\op \to \Set$. Equivalently, a paracyclic set can be defined as a simplicial set $X_\bullet$, with face maps $d_i^n: X_n \to X_{n-1}$ and degeneracy maps $s_i^n: X_n \to X_{n+1}$, equipped with isomorphisms $\tau^n: X_n \to X_n$ satisfying relations dual to \eqref{eqn:paracycliccoface}--\eqref{eqn:paracycliccodegen}:
\begin{align}\label{eqn:paracyclicface}
    d_i^n \tau^n &= \begin{cases}
        \tau^{n-1} d_{i+1}^n, & i<n,\\
        d_0^n, & i=n,
    \end{cases}\\ \label{eqn:paracyclicdegen}
    s_i^n \tau^n &= \begin{cases}
        \tau^{n+1} s_{i+1}^n, & i<n,\\
        (\tau^{n+1})^2 s_0^n, & i=n.
    \end{cases}
\end{align}
Following the discussion in Section \ref{sec:paracycliccat}, a paracyclic set possesses extra degeneracy maps $s_{n+1}^n: X_n \to X_{n+1}$, which can be defined as $s_{n+1}^n = \tau^{n+1}s_0^n$, and the simplicial relations extend with the exceptional rule $d_0^{n+1} s_{n+1}^n = \tau^n$.

On the other hand, suppose that $X_\bullet$ is a simplicial set equipped with extra degeneracy maps $s_{n+1}^n$ satisfying relations dual to \eqref{eqn:paracycliclift1}--\eqref{eqn:paracycliclift2}:
\begin{align}\label{eqn:paracyclicgamma1}
    d_i^{n+1} s_{n+1}^n &= \begin{cases}
        s_n^{n-1} d_i^n, & 0 < i < n+1, \\
        \id, & i = n+1,
    \end{cases}\\ \label{eqn:paracyclicgamma2}
    s_i^{n+1} s_{n+1}^n &= s_{n+2}^{n+1} s_i^n, \;\;\;\;\;\; 0 \leq i \leq n+1.
\end{align}
Defining $\tau^n: X_n \to X_n$ by $\tau^n = d_0^{n+1} s_{n+1}^n$, the relations \eqref{eqn:paracyclicface}--\eqref{eqn:paracyclicdegen} can be deduced from \eqref{eqn:paracyclicgamma1}--\eqref{eqn:paracyclicgamma2}. If we additionally verify that the maps $\tau^n$ are isomorphisms, then we can conclude that they give a paracyclic structure on $X_\bullet$.

\subsection{Frobenius pseudomonoids in \texorpdfstring{$\Span$}{Span}}\label{sec:frobeniusspan}

This subsection is devoted to the proof of the following result.

\begin{thm}\label{thm:paracyclic}
    Let $X_\bullet$ be a $2$-Segal set. Then there is a one-to-one correspondence between paracyclic structures on $X_\bullet$ and equivalence classes of Frobenius structures on the corresponding pseudomonoid in $\Span$.
\end{thm}

\subsubsection*{A paracyclic structure determines a Frobenius structure}

We begin with the following Lemma, the proof of which is formally identical to the proof of unitality in \cite{FGKW:unital}.
\begin{lem}\label{lem:extraunital}
    Let $X_\bullet$ be a paracyclic set whose underlying simplicial set is $2$-Segal. Then the diagram
\begin{equation}\label{eqn:paraunital}
    \begin{tikzcd}
    X_1 \arrow[r, "d_1"] \arrow[d, "s_2"] & X_0 \arrow[d, "s_1"] \\
    X_2 \arrow[r, "d_1"] & X_1
\end{tikzcd}
\end{equation}
is a pullback.
\end{lem}

Suppose that $X_\bullet$ is a $2$-Segal paracyclic set. Then we may form the span
\[ 
\begin{tikzcd}
	X_1 & \arrow[l,"s_1"'] X_0 \arrow[r] & \pt,
\end{tikzcd}
\]
which will be the counit morphism $\varepsilon$ for a Frobenius structure on the pseudomonoid in $\Span$ constructed in Section \ref{sec:pseudospan}. Using Lemma \ref{lem:extraunital}, we find that the induced pairing $\alpha = \varepsilon \circ \mu$ is canonically isomorphic to
\[ 
\begin{tikzcd}
	(X_1)^2 & X_2 \arrow[l,"{(d_2,d_0)}"'] & X_1 \arrow[l,"s_2"'] \arrow[r,"d_1"] & X_0 \arrow[r] & \pt.
\end{tikzcd}
\]
Since both $d_2^2 s_2^1 = \id$ and $d_0^2 s_2^1 = \tau^1$ are isomorphisms, it follows that $\alpha$ is biexact, so $\varepsilon$ gives a Frobenius structure.

\subsubsection*{A Frobenius structure determines a paracyclic structure}

Suppose that $X_\bullet$ is a $2$-Segal set for which the corresponding pseudomonoid in $\Span$ is equipped with a Frobenius structure, i.e.\ a counit $\varepsilon: X_1 \to \pt$
such that $\alpha = \varepsilon \circ \mu : X_1 \times X_1 \to \pt$ is biexact. By \cite[Lemma 3.4]{Stern:span} (also see \cite[Corollary 4.2]{ContrerasMehtaSpan} for the $1$-categorical analogue, which directly extends to the present context), it follows that $\alpha$ is uniquely isomorphic to
\begin{equation}\label{eqn:alpha}
\begin{tikzcd}
X_1 \times X_1 & \arrow[l, "{(\id, \tau^1)}"'] X_1 \arrow[r] & \pt
\end{tikzcd}
\end{equation}
where $\tau^1: X_1 \to X_1$ is an automorphism.

Define $s_1^0 = \tau^1 s_0^0: X_0 \to X_1$.
\begin{lem}\label{lem:counit}
    The counit $\varepsilon$ is uniquely isomorphic to the span
    \[
    \begin{tikzcd}
        X_1 & \arrow[l, "s_1^0"'] X_0 \arrow[r] & \pt.
    \end{tikzcd}\]
\end{lem}
\begin{proof}
    The existence of the isomorphism is shown in \cite{ContrerasMehtaSpan}[Lemma 4.5]. Uniqueness follows from the fact that $s_1^0$ is injective.
\end{proof}

By the transitivity of 2-isomorphism, two Frobenius structures are equivalent if and only if their associated maps $s^0_1$ given by Lemma \ref{lem:counit} are equal.

\begin{lem}\label{lem:paracyclicpullback}
\begin{enumerate}
    \item $d_1^1 s_1^0 = \id$,
    \item There is a unique map $s_2^1: X_1 \to X_2$ such that
        \begin{itemize}
            \item $d_0^2 s_2^1 = \tau^1$,
            \item $d_1^2 s_2^1 = s_1^0 d_1^1$,
            \item $d_2^2 s_2^1 = \id$,
        \end{itemize}
    and where the diagram \eqref{eqn:paraunital} is a pullback.
\end{enumerate}
\end{lem}
\begin{proof}
    Consider the equation $\alpha = \varepsilon \circ \mu$. Using \eqref{eqn:alpha} and Lemma \ref{lem:counit}, we have that there are unique maps $X_1 \to X_2$ and $X_1 \to X_0$ such that the diagram
    \begin{equation}\label{eqn:counitpairing}
        \begin{tikzcd}
            & & X_1 \arrow[dl] \arrow[dr] \arrow[ddll, "{(\id,\tau^1)}"', bend right] && \\
            & X_2 \arrow[dl, "{(d_2,d_0)}"] \arrow[dr, "d_1"] && X_0 \arrow[dl,"s_1"'] \arrow[dr] \\
            (X_1)^2 && X_1 && \pt
        \end{tikzcd}
    \end{equation}
    commutes, and where the middle square is a pullback. We define $s_2^1$ to be the map from $X_1$ to $X_2$ in \eqref{eqn:counitpairing}. The second part of the lemma will follow if we can show that the map from $X_1$ to $X_0$ in \eqref{eqn:counitpairing} is $d_1$. Before doing this, we will prove the first part of the lemma.

    For $u \in X_0$, let $\psi = s_2 s_0 u$. From the commutativity of \eqref{eqn:counitpairing}, we have $d_2^2 s_2^1 = \id$ and $d_0^2 s_2^1 = \tau^1$, so it follows that
    \begin{align*}
        d_2 \psi &= s_0 u,\\
        d_0 \psi &= \tau s_0 u = s_1 u.
    \end{align*}
    Using the above equations and the simplicial identities, we then have $d_1 s_1 u = d_1 d_0 \psi = d_0 d_2 \psi = d_0 s_0 u = u$, which proves the first part of the lemma.

    It remains to show that the map from $X_1$ to $X_0$ in \eqref{eqn:counitpairing} is $d_1$. For $x \in X_1$, let $u$ be its image in $X_0$ under said map. Since $s_1^0$ is injective, we have that $u$ is uniquely determined by the equation $s_1 u = d_1 s_2 x$. Taking $d_1$ of both sides and using the first part of the lemma, we have $u = d_1 d_1 s_2 x = d_1 d_2 s_2 x = d_1 x$.
\end{proof}

\begin{lem}\label{lem:ss}
    \begin{enumerate}
        \item $s_0^1 s_1^0 = s_2^1 s_0^0$,
        \item $s_1^1 s_1^0 = s_2^1 s_1^0$.
    \end{enumerate}
\end{lem}
\begin{proof}
    For $u \in X_0$, let $\psi = s_0 s_1 u$. Then $d_1 \psi = s_1 u$, so by the pullback property in Lemma \ref{lem:paracyclicpullback}, there exists a unique $x \in X_1$ such that $\psi = s_2 x$. Applying $d_2$ to both sides and using the simplicial relations together with the relations in Lemma \ref{lem:paracyclicpullback}, we get $x = d_2 \psi = d_2 s_0 s_1 u = s_0 d_1 s_1 u = s_0 u$. Thus, $\psi = s_2 s_0 u$, which gives the first identity. The proof of the second identity is similar.
\end{proof}

For $n \geq 2$, we define $s_{n+1}^n : X_n \to X_{n+1}$ in a similar way to the description of the degeneracy maps in Section \ref{sec:faceanddegen}. Given $\psi \in X_n$, $n \geq 2$, viewed as an $(n+1)$-gon, we obtain $s_{n+1} \psi$ by attaching the ``degenerate'' $2$-simplex $s_2 e_\out \psi$ along the edge $e_\out \psi$; see Figure \ref{fig:extradegen}.

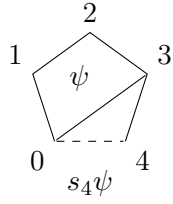
\begin{figure}[th]
\begin{center}
\begin{tikzpicture}[scale=0.8]
    \begin{scope}
        \draw (-54: 1) node[vertex] (4) {} node[anchor=north west] {4}
        -- (18:1) node[vertex] (3) {} node[anchor=south west] {3}
        -- (90:1) node[vertex] (2) {} node[anchor=south] {2}
        -- (162:1) node[vertex] (1) {} node[anchor=south east] {1}
        -- (-126:1) node[vertex] (0) {} node[anchor=north east] {0};
        \draw[dashed] (0) -- (4);
        \draw (0) -- (3);
        \node at (126:.3) {$\psi$};
        \node at (0,-1.5) {$s_4 \psi$};
    \end{scope}

\end{tikzpicture}
\end{center}
    \caption{Graphical calculus for the extra degeneracy maps. The $2$-simplex appended to $\psi$ is $s_2 e_\out \psi$. The edge from $0$ to $4$ is dashed to emphasize that it is in the image of $s_1^0$.}
    \label{fig:extradegen}
\end{figure}

\begin{lem}\label{lem:paracyclicrelations}
    The relations \eqref{eqn:paracyclicgamma1} hold for all $n \geq 2$, and the relations \eqref{eqn:paracyclicgamma2} hold for all $n \geq 1$.
\end{lem}
\begin{proof}
    The graphical calculus makes it straightforward to see that \eqref{eqn:paracyclicgamma1} holds for $0 < i < n$ and for $i=n+1$. For $i=n$, the situation can be reduced to the case $n=2$, since the maps involved only affect the $2$-simplex with vertices $0$, $n-1$, and $n$. 
    
\begin{figure}[th]
\begin{center}
\begin{tikzpicture}[scale=0.8]
    \begin{scope}[shift={(0,0)}]
        \draw (-45:1) node[vertex] (3) {} node[anchor=north west] {3}
        -- (45:1) node[vertex] (2) {} node[anchor=south west] {2}
        -- (135:1) node[vertex] (1) {} node[anchor=south east] {1}
        -- (-135:1) node[vertex] (0) {} node[anchor=north east] {0};
        \draw[dashed] (0) -- (3);
        \draw (0) -- (2);
        \node at (135:.4) {$\psi$};
        \node at (-45:.4) {${\scriptstyle s_2 d_1}$};
    \end{scope}
    \equals{2}{0}
    \node at (2,-1.5) {$s_3 \psi$};    
    \begin{scope}[shift={(4,0)}]
        \draw (-45:1) node[vertex] (3) {} node[anchor=north west] {3}
        -- (45:1) node[vertex] (2) {} node[anchor=south west] {2}
        -- (135:1) node[vertex] (1) {} node[anchor=south east] {1}
        -- (-135:1) node[vertex] (0) {} node[anchor=north east] {0};
        \draw[dashed] (0) -- (3);
        \draw (1) -- (3);
        \node at (-135:.4) {${\scriptstyle d_2 s_3}$};
    \end{scope}
    \node at (7,0) {$\xmapsto[\phantom{xxxxxxx}]{d_2}$};
    \begin{scope}[shift={(10,-.3)}]
        \draw (-30:1) node[vertex] (2) {} node[anchor=north west] {2}
        -- (90:1) node[vertex] (1) {} node[anchor=south] {1}
        -- (210:1) node[vertex] (0) {} node[anchor=north east] {0};
        \draw[dashed] (0) -- (2);
        \node at (0,0) {${\scriptstyle d_2 s_3 \psi}$};
    \end{scope}        
\end{tikzpicture}
\end{center}
    \caption{Constructing $d_2 s_3 \psi$ for $\psi \in X_2$.}
    \label{fig:d2s3}
\end{figure}
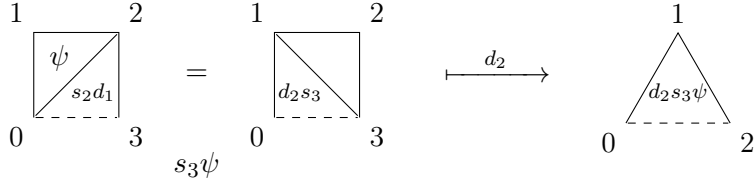

    For $\psi \in X_2$, the process of constructing $d_2 s_3 \psi$ involves attaching the $2$-simplex $s_2 d_1 \psi$, applying the $2$-Segal map $\hat{\mathcal{T}}_{02} \hat{\mathcal{T}}_{13}^{-1}$ to change the triangulation, and then deleting the $2$-simplex at vertex $2$. This is illustrated in Figure \ref{fig:d2s3}. Let $\xi = d_2 s_3 \psi$. Since $d_1 \xi$ is in the image of $s_1$, it follows from the pullback property in Lemma \ref{lem:paracyclicpullback} that there exists a unique $x \in X_1$ such that $s_2 x = \xi$. Applying $d_2$ to both sides, we get $x = d_2 \xi = d_2 \psi$, where in the last step we used the fact that the $2$-Segal map fixes the boundary edges. Thus, $d_2 s_3 \psi = \xi = s_2 d_2 \psi$, and it follows that \eqref{eqn:paracyclicgamma1} holds for $i=n$.

    The graphical calculus makes it straightforward to see that \eqref{eqn:paracyclicgamma2} holds for $i < n+1$. For $i=n+1$, the situation can be reduced to the case $n=1$, since the maps involved only depend on the edge $e_\out$. The proof for this case is similar to the analysis in Figure \ref{fig:d2s3}, and we leave it as an exercise.
\end{proof}

Lemmata \ref{lem:paracyclicpullback}, \ref{lem:ss}, and \ref{lem:paracyclicrelations} cover all the cases of relations \eqref{eqn:paracyclicgamma1}--\eqref{eqn:paracyclicgamma2}. The following lemma provides the remaining result needed to conclude that the maps $s_{n+1}^n$ give a paracyclic structure. 

\begin{lem}
The maps $\tau^n := d_0^{n+1}s_{n+1}^n$ are invertible.
\end{lem}
\begin{proof}
    The proof for $n \geq 2$ can be reduced to the $n=2$ case, since the maps involved only affect the $2$-simplex with vertices $0$, $1$, and $n$. For $\psi \in X_2$, the process of constructing $\tau \psi = d_0 s_3 \psi$ involves attaching the $2$-simplex $s_2 d_1 \psi$, applying the $2$-Segal map $\hat{\mathcal{T}}_{02} \hat{\mathcal{T}}_{13}^{-1}$ to change the triangulation, and then deleting the $2$-simplex at vertex $0$. This is illustrated in Figure \ref{fig:d0s3}. 

\begin{figure}[th]
\begin{center}
\begin{tikzpicture}[scale=0.8]
    \begin{scope}[shift={(0,0)}]
        \draw (-45:1) node[vertex] (3) {} node[anchor=north west] {3}
        -- (45:1) node[vertex] (2) {} node[anchor=south west] {2}
        -- (135:1) node[vertex] (1) {} node[anchor=south east] {1}
        -- (-135:1) node[vertex] (0) {} node[anchor=north east] {0};
        \draw[dashed] (0) -- (3);
        \draw (0) -- (2);
        \node at (135:.4) {$\psi$};
        \node at (-45:.4) {${\scriptstyle s_2 d_1}$};
    \end{scope}
    \equals{2}{0}
    \node at (2,-1.5) {$s_3 \psi$};    
    \begin{scope}[shift={(4,0)}]
        \draw (-45:1) node[vertex] (3) {} node[anchor=north west] {3}
        -- (45:1) node[vertex] (2) {} node[anchor=south west] {2}
        -- (135:1) node[vertex] (1) {} node[anchor=south east] {1}
        -- (-135:1) node[vertex] (0) {} node[anchor=north east] {0};
        \draw[dashed] (0) -- (3);
        \draw (1) -- (3);
        \node at (45:.4) {${\scriptstyle \tau \psi}$};
    \end{scope}
    \node at (7,0) {$\xmapsto[\phantom{xxxxxxx}]{d_0}$};
    \begin{scope}[shift={(10,-.3)}]
        \draw (-30:1) node[vertex] (2) {} node[anchor=north west] {2}
        -- (90:1) node[vertex] (1) {} node[anchor=south] {1}
        -- (210:1) node[vertex] (0) {} node[anchor=north east] {0}
        -- cycle;        
        \node at (0,0) {$\tau \psi$};
    \end{scope}        
\end{tikzpicture}
\end{center}
    \caption{Constructing $\tau \psi$ for $\psi \in X_2$.}
    \label{fig:d0s3}
\end{figure}
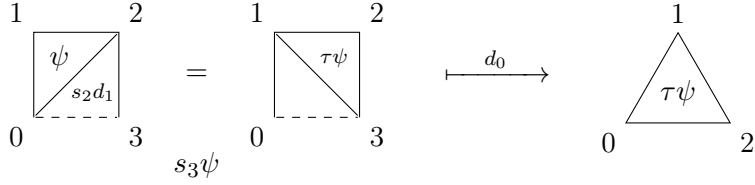
The invertibility of this process amounts to the fact that the deleted $2$-simplex $d_2 s_3 \psi$ can be reconstructed from $\tau \psi$. Specifically, using the fact that $\tau^1$ is invertible, we have $d_2 s_3 \psi = s_2 d_2 \psi = s_2 \tau^{-1} \tau d_2 \psi = s_2 \tau^{-1} d_1 \tau \psi$. 

Finally, it can be shown using the paracyclic relations that $d_1^1 (\tau^1)^{-1} s_0^0$ is the inverse of $\tau^0$.
\end{proof}

To summarize this section so far, we have shown that a paracyclic structure on a $2$-Segal set $X_\bullet$ induces a Frobenius structure on the corresponding pseudomonoid in $\Span$, and conversely, that a Frobenius structure on the pseudomonoid induces a paracyclic structure. To complete the proof of Theorem \ref{thm:paracyclic}, we observe that the two directions are inverses (up to equivalence). This follows from the fact that the construction of $s_2^1$ in Lemma \ref{lem:paracyclicpullback} is unique, and that the higher maps $s_{n+1}^n$ are uniquely determined by $s_2^1$.

\begin{rmk}
Recall from Section \ref{sec:paracycliccat} that there is a natural (full) functor from the paracyclic category $\Lambda_\infty$ to the cyclic category $\Lambda$. Thus we can view cyclic sets, i.e.\ functors $\Lambda^\op \to \Set$, as paracyclic sets that factor through $\Lambda$. In terms of the generator- and-relation description in Section \ref{sec:paracyclicset}, a paracyclic set is cyclic if $(\tau^n)^{n+1} = \id$ for all $n$.
 
Given a 2-Segal paracyclic set $X_\bullet$, one can use the 2-Segal conditions to show that $X_\bullet$ is cyclic if and only if $(\tau^1)^2=\id$ and $(\tau^2)^3=\id$. The one necessary relation, $\tau^0=\id$, which does not follow from the 2-Segal conditions follows directly from the paracyclic identities since 
\[
\tau^0=d^1_0\circ s_0^0\circ \tau^0= d_0^1\circ (\tau^1)^2\circ s^0_0=d_0^1\circ s_0^0=\id.
\] 
In \cite{Stern:span}, it was shown (in a more general setting than ours) that $2$-Segal cyclic objects correspond to Calabi-Yau objects (a categorification of symmetric Frobenius objects) in $\Span$. Thus we see that a Frobenius pseudomonoid in $\Span$ is in fact Calabi-Yau if it satisfies the symmetry condition $(\tau^1)^2 = \id$ and the coherence condition $(\tau^2)^3 = \id$.

\end{rmk}

\subsection{Examples}\label{subsec:para_egs}

\begin{example}[Groupoids]
Let $\mathcal{G}$ be a groupoid. As a special case of Example \ref{ex:categories}, the nerve $N_\bullet \mathcal{G}$ is $2$-Segal. We can define additional degeneracy maps by $s_1^0(u) = s_0^0(u) = \id_u$ and
\[ s_{n+1}^n(g_1, \dots, g_n) = (g_1, \dots, g_n, (g_1 \cdots g_n)^{-1})\]
for $n \geq 1$. This gives a cyclic structure where the associated automorphisms $\tau^n: N_n\mathcal{G} \to N_n\mathcal{G}$ are given by
\[ \tau^n(g_1, \dots, g_n) = (g_2, \dots, g_n, (g_1 \cdots g_n)^{-1}).\]

More generally, suppose that $\mathcal{G}$ is equipped with a bisection, i.e.\ a map $\omega: N_0 \mathcal{G} \to N_1 \mathcal{G}$ such that $d_1 \omega = \id$ and $d_0 \omega$ is bijective. Then we can take $s_1^0 = \omega$ and
\[ s_{n+1}^n(g_1, \dots, g_n) = (g_1, \dots, g_n, (g_1 \cdots g_n)^{-1}\omega( d_1(g_1)))\]
for $n \geq 1$. This gives a paracyclic structure where the associated automorphisms $\tau^n$ are given by
\[ \tau^n(g_1, \dots, g_n) = (g_2, \dots, g_n, (g_1 \cdots g_n)^{-1}\omega( d_1(g_1))).\]
This paracyclic structure is cyclic if and only if $\omega$ is central, in the sense that $\omega(d_1(g))^{-1} g \omega(d_0(g)) = g$ for all $g \in \mathcal{G}$.
\end{example}

\begin{example}[A partial monoid example]
Here is a simple example of a partial monoid for which the nerve (see Example \ref{ex:partialmonoids}) admits a cyclic structure. For a fixed natural number $L$, consider the set $M = \{0, \dots, L\}$ with the partial monoid structure given by addition, where the operation is undefined for sums larger than $L$. Then the nerve $N_\bullet M$ has
$N_0 M = \{0\}$ and
\[ N_n M = \left\{(x_1, \dots, x_n) \in M^n \suchthat \sum x_i \leq L\right\}\]
for $n \geq 1$.

We define additional degeneracy maps by $s_1^0(0) = L$ and
\[ s_{n+1}^n(x_1, \dots, x_n) = (x_1, \dots, x_n, L - \sum x_i)\]
for $n \geq 1$. This gives a cyclic structure where the associated automorphisms $\tau^n$ are given by 
\[ \tau^n(x_1,\dots, x_n) = (x_2, \dots, x_n, L - \sum x_i).\]

This example fits into a larger class of examples of $2$-Segal cyclic sets, which are associated to \emph{effect algebroids}. In \cite[Theorem 5.1.4]{Roumen_2017}, Roumen shows that there is a fully faithful embedding of effect algebroids into 2-Segal cyclic sets. In this example, the partial monoid $M$ is an effect algebra, with orthocomplement $a^\perp=L-a$.
\end{example}

\begin{example}[Twisted cyclic nerves]
Let $\C$ be a small category equipped with an automorphism $F: \C \to \C$. Then the twisted cyclic nerve $N^F_\bullet \C$ (see Example \ref{ex:twistedcyclic}) has additional degeneracy maps $s_{n+1}^n$ that insert an identity in the first entry. This gives a paracyclic structure where the associated automorphisms $\tau^n$ are given by
\[ \tau^n(f_n, \dots, f_0) = (F(f_0), f_n, \dots, f_1).\]
We note that invertibility of $F$ is needed to ensure the invertibility of $\tau^n$. This paracyclic structure is only cyclic in the case where $F = \id$.

As special cases, one can obtain paracyclic structures on the twisted inertia groupoid associated to a group equipped with an automorphism (see Example \ref{ex:twistedinertia}), and on the building associated to a poset equipped with an automorphism (see Example \ref{ex:buildings}).
\end{example}

\section{\texorpdfstring{$\Gamma$}{Gamma}-structures and commutative pseudomonoids in \texorpdfstring{$\Span$}{Span}}
\label{sec:gammacommutative}

In this section, we consider commutative structures on pseudomonoids in $\Span$. We find that commutative structures correspond to \emph{$\Gamma$-structures} on $2$-Segal sets.

In Section \ref{sec:finitepointed}, we provide some background information on the category (which we denote $\Phi_\ast$) of finite pointed cardinals, which is a skeleton of the category of finite pointed sets. It is also opposite to Segal's category $\Gamma$, which first appeared in \cite{SegalCoh} as a tool to study infinite loop spaces, and is now a significant part of in the Connes-Consani approach to the field with one element \cite{Connes-Consani:Absolute}. 

A \emph{$\Gamma$-set} is a functor $\Gamma^{\op} = \Phi_\ast \to \Set$. In Section \ref{sec:cut}, we describe a functor $\Cut: \Delta^\op \to \Phi_\ast$. Via $\Cut$, we can obtain from any $\Gamma$-set an underlying simplicial set, so one can think of a $\Gamma$-set as a simplicial set with some additional structure. 

The next parts of the section are devoted to explicitly describing this additional structure. In Sections \ref{sec:fincard} and \ref{subsec:gen_rel_Phistar}, we use a generator-and-relation description in \cite{Grandis} of the category $\Phi$ of finite cardinals to obtain a generator-and-relation description of $\Phi_\ast$. In Section \ref{sec:gamma}, we arrive at Theorem \ref{thm:finstar}, which says that a $\Gamma$-set is equivalent to a simplicial set $X_\bullet$ where each set $X_n$ is equipped with an action of the symmetric group $S_n$, satisfying certain compatibility relations.

In Section \ref{sec:gammacommutative}, we prove the second main result of the paper, that a $\Gamma$-structure on a $2$-Segal set is equivalent to a commutative structure on the corresponding pseudomonoid in $\Span$. We conclude the section with examples in Section \ref{subsec:sym_egs}.

\subsection{The category of finite pointed cardinals}\label{sec:finitepointed}

Let $\Phi_\ast$ denote the category of finite pointed cardinals. Its objects are the sets
\[\fin{n} = \{\ast\} \cup \{1,\dots,n\}\]
for $n \geq 0$, and its morphisms are maps $f: \fin{n} \to \fin{m}$ such that $f(\ast) = \ast$.

The category $\Phi_\ast$ appears in relation to commutative algebraic structures in numerous ways. First, it is a skeleton of the category $\Fin_\ast$ of finite pointed sets, which is the category of operators for the commutative operad. Additionally, $\Phi_\ast$ is equivalent to the opposite category of \emph{Segal's category} $\Gamma$, which is closely related to commutative algebraic objects in higher categories. In particular, in \cite{SegalCoh}, Segal defines a $\Gamma$-space to be a functor
\[
\begin{tikzcd}
A:&[-3em] \Gamma^\op \cong \Phi_\ast \arrow[r] & \scr{S} 
\end{tikzcd}
\]
(where $\scr{S}$ is the category of topological spaces) such that $A(\fin{0})$ is contractible, and where the maps 
\[
A(\fin{n} )\to A(\fin{1})\times \cdots \times A(\fin{1}) 
\]
induced by the morphisms 
\[
\begin{tikzcd}[ampersand replacement=\&,row sep=0em]
\rho^i: \&[-3em] \fin{n} \arrow[r] \& \fin{1} \\
 \& j \arrow[r,mapsto] \& \begin{cases}
 	1 & \mbox{ if }j=i,\\
 	\ast & \mbox{ otherwise}
 \end{cases}
\end{tikzcd}
\]
are homotopy equivalences. Segal's $\Gamma$-spaces are one model for $E_\infty$-spaces, and thus for infinite loop spaces (the latter under additional conditions). 

\subsection{The \texorpdfstring{$\Cut$}{Cut} functor}\label{sec:cut}

There is a functor $\Cut: \Delta^\op \to \Phi_\ast$, defined as follows. On objects, we have $\Cut([n]) = \fin{n}$. Given a morphism $f \in \Hom_{\Delta}([n],[m])$, the induced morphism $\Cut(f) \in \Hom_{\Phi_\ast}(\fin{m},\fin{n})$ is defined by $\Cut(f)(\ast) = \ast$ and, for $i = 1, \dots, m$,
\[ \Cut(f)(i) = 
\begin{cases} 
\ast & \mbox{if } f(0) \geq i \mbox{ or } f(n) < i, \\
\min\{k : f(k) \geq i\} & \mbox{ otherwise}.
\end{cases} \]
One can visualize the Cut functor by depicting the numerical elements of $\fin{n}$ as representing the interstices between the elements of $[n]$, with $\ast$ representing the exterior regions (see Figure \ref{fig:cutobject}).
        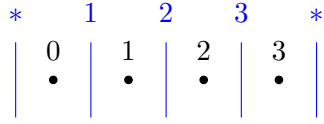
\begin{figure}
	\begin{center}
		\begin{tikzpicture}
		\foreach \x/\lab in {0/0,1/1,2/2,3/3}{
		\path (\x,0) node[label=above:$\lab$]{};
		\draw[fill=black] (\x,0) circle (0.05);        
		};
		\foreach \x in {-0.5,0.5,1.5,2.5,3.5}{
		\draw[blue] (\x,-0.5) -- (\x,0.5);
		} 
		\foreach \x/\lab in {1/1,2/2,3/3}{
		\path[blue] (\x-0.5,0.5) node[label=above:$\lab$]{};   
		};
		\path[blue] (-0.5,0.5) node[label=above:$\ast$]{};  
		\path[blue] (3.5,0.5) node[label=above:$\ast$]{};    
		\end{tikzpicture}
	\end{center}
        \caption{Elements of $\fin{3}$ depicted as interstices of the elements of $[3]$.}
        \label{fig:cutobject}
        \end{figure}
        
A morphism $f \in \Hom_\Delta([n],[m])$ can be depicted by arrows from each $k \in [n]$ to $f(k) \in [m]$. Since $f$ is monotonic, the arrows are noncrossing. Thus, for each $i \in \fin{m}$, the interstice representing $i$ can be connected to a unique interstice representing $\Cut(f)(i) \in \fin{n}$ or to the exterior region. An example is shown in Figure \ref{fig:cutmorphism}.
	\begin{figure}[t]
            \begin{center}
		\begin{tikzpicture}[yscale=.8]
		\foreach \x/\lab in {0/0,1/1,2/2,3/3}{
			\path (\x,0) node[label=above:$\lab$] (A\lab){};
			\draw[fill=black] (\x,0) circle (0.05);
		};
		\foreach \x in {-0.5,0.5,1.5,2.5,3.5}{
			\draw[blue] (\x,-0.5) -- (\x,0.5);
		} 
		\foreach \x/\lab in {1/1,2/2,3/3}{
		\path[blue] (\x-0.5,0.5) node[label=above:$\lab$]{};   
		};
		\path[blue] (-0.5,0.5) node[label=above:$\ast$]{};  
		\path[blue] (3.5,0.5) node[label=above:$\ast$]{};  
  
		\foreach \x/\lab in {0.5/0,1.5/1,2.5/2}{
			\path (\x,-3) node[label=below:$\lab$] (B\lab){};
			\draw[fill=black] (\x,-3) circle (0.05);
		};
		\foreach \x in {0,1,2,3}{
			\draw[blue] (\x,-3.5) -- (\x,-2.5);
		};
	\foreach \x/\lab in {1/1,2/2}{
		\path[blue] (\x,-3.5) node[label=below:$\lab$]{};   
		\path[blue] (0,-3.5) node[label=below:$\ast$]{};  
		\path[blue] (3,-3.5) node[label=below:$\ast$]{};     
		} 
	  \draw[->] (B0) to (A0);
	  \draw[->] (B1) to (A2);
	  \draw[->] (B2) to (A2);
	  \begin{scope}[color=red,decoration={
	  	markings,
	  	mark=at position 0.5 with {\arrow{>}}}]
	  	\draw[postaction={decorate}] (-0.5,-0.5) to (0,-2.5);
	  	\draw[postaction={decorate}] (0.5,-0.5) to (1,-2.5);
	  	\draw[postaction={decorate}] (1.5,-0.5) to (1,-2.5);
	  	\draw[postaction={decorate}] (2.5,-0.5) to (3,-2.5);
	  	\draw[postaction={decorate}] (3.5,-0.5) to (3,-2.5);
	  \end{scope}
		\end{tikzpicture}
        \end{center}
        \caption{A morphism $f \in \Hom_\Delta([2],[3])$, where $f(0) = 0$, $f(1) = f(2) = 2$. The induced morphism $\Cut(f) \in \Hom_{\Phi_\ast}(\fin{3},\fin{2})$ is given by $\Cut(f)(\ast) = \Cut(f)(3) = \ast$, $\Cut(f)(1) = \Cut(f)(2) = 1$.}
        \label{fig:cutmorphism}
        \end{figure}

\subsection{The category of finite cardinals}\label{sec:fincard}

It is often useful to describe simplicial sets in terms of face and degeneracy maps. Similarly, it will be useful to have a description of $\Phi_\ast$ in terms of generators and relations, in a way that is compatible with the functor $\Cut$. To derive such a description, we will make use of the well-studied category $\Phi$ of finite cardinals. 

The category $\Phi$ is the full subcategory of $\Set$ with objects $\underline{0} = \emptyset$ and $\underline{n} = \{0, \dots, n-1\} = [n-1]$ for $n \geq 1$. In \cite{Grandis}, Grandis obtained a generator-and-relation description of $\Phi$, which we summarize in this subsection. In Section \ref{subsec:gen_rel_Phistar}, we will use the description of $\Phi$ to obtain a generator-and-relation description of $\Phi_\ast$.

The simplex category $\Delta$ is a subcategory of $\Phi$. We therefore have the coface maps $\delta_i^n: \underline{n} \to \underline{n+1}$ for $0<n$ and $0\leq i\leq n$ and codegeneracy maps $\sigma_i^n: \underline{n+2} \to \underline{n+1}$ for $0 \leq i \leq n$, satisfying the cosimplicial relations
\begin{align}\label{eqn:cosimp1}
    \delta_i \delta_j &= \delta_{j+1} \delta_i, \;\;\;\;\;\; i \leq j,\\
    \sigma_j \sigma_i &= \sigma_i \sigma_{j+1}, \;\;\;\;\;\; i \leq j, \\
    \sigma_j \delta_i &= \begin{cases}
        \delta_i \sigma_{j-1}, & i < j, \\
        \id, & i = j, j+1, \\
        \delta_{i-1} \sigma_j, & i > j+1.
    \end{cases} \label{eqn:cosimp3}
\end{align}
In $\Phi$, there are also \emph{main transposition maps} $r_i^n: \underline{n+2} \to \underline{n+2}$ for $0 \leq i \leq n$ which exchange $i$ and $i+1$. The main transposition maps generate the permutation groups and satisfy the \emph{Moore relations}
\begin{align} \label{eqn:moore1}
    (r_i)^2 &= \id, \\
    r_i r_j r_i &= r_j r_i r_j, \;\;\; i = j-1, \\
    r_i r_j &= r_j r_i, \;\;\;\;\;\; i < j-1. \label{eqn:moore3}
\end{align}

An arbitrary map $g: \underline{n} \to \underline{m}$ can be factored as $g = h \rho$, where $\rho: \underline{n} \to \underline{n}$ is a permutation and $h: \underline{n} \to \underline{m}$ is monotonic, so the coface, codegeneracy, and main transposition maps generate $\Phi$. In addition to the cosimplicial and Moore relations, there are the following \emph{mixed relations}:
\begin{align}\label{eqn:mixed1}
    r_i \delta_j &= \begin{cases}
        \delta_j r_i, & i < j-1, \\
        \delta_{i}, & i = j-1, \\
        \delta_{i+1}, & i = j, \\        
        \delta_j r_{i-1}, & i>j,
    \end{cases} \\
    r_i \sigma_j &= \begin{cases}
        \sigma_j r_i, & i < j-1,\\
        \sigma_i r_{i+1} r_i, & i = j-1,\\
        \sigma_{i+1}r_i r_{i+1}, & i=j,\\
        \sigma_j r_{i+1}, & i>j,
    \end{cases} \\
    \sigma_i r_i &= \sigma_i.\label{eqn:mixed3}
\end{align}
In \cite{Grandis}, it is shown that the above relations are sufficient, so that $\Phi$ is the category generated by the coface, codegeneracy, and main transposition maps under the relations \eqref{eqn:cosimp1}--\eqref{eqn:mixed3}.

\subsection{A generator-and-relation description of \texorpdfstring{$\Phi_\ast$}{the category of finite pointed cardinals}}\label{subsec:gen_rel_Phistar}

There is a forgetful functor $P: \Phi_\ast \to \Phi$ which, on objects, takes $\fin{n}$ to $\underline{n+1}$, and on morphisms, takes $f \in \Hom_{\Phi_\ast}(\fin{n},\fin{m})$ to $P(f): \underline{n+1} \to \underline{m+1}$, given by $P(f)(0) = 0$ and
\[ P(f)(i) = \begin{cases}
    0 & \mbox{if } f(i) = \ast, \\
    f(i) & \mbox{otherwise}
\end{cases}\]
for $i = 1,\dots,n$.

The functor $P$ is faithful; its image consists of maps $g: \underline{n} \to \underline{m}$ with $n,m > 0$ and such that $g(0) = 0$. In particular, the image of $P$ contains the coface maps $\delta_i^n$ for $n \geq 1$ and $1 \leq i \leq n$, all of the codegeneracy maps $\sigma_i^n$, and the main transpositions $r_i^n$ for $1 \leq i \leq n$. We now consider the preimages of these maps in $\Phi_\ast$.

For $0 \leq i \leq n$, let $s_i^n \in \Hom_{\Phi_\ast}(\fin{n},\fin{n+1})$ be given by $P(s_i^n) = \delta_{i+1}^{n+1}$. Explicitly, this is the map that skips $i+1$:
\[s_i^n (k) = \begin{cases}
    k & \mbox{if } 1 \leq k \leq i,\\
    k+1 & \mbox{if } i < k \leq n.
\end{cases}
\]

For $0 \leq i \leq n-1$, let $d_i^n \in \Hom_{\Phi_\ast}(\fin{n}, \fin{n-1})$ be given by $P(d_i^n) = \sigma_i^{n-1}$. For $i > 0$, this is the map that collapses $i$ and $i+1$:
\[ d_i^n(k) = \begin{cases}
    k & \mbox{if } 1 \leq k \leq i,\\
    k-1 & \mbox{if } i < k \leq n.
\end{cases}
\]
Additionally, $d_0^n$ collapses $\ast$ and $1$. Specifically, $d_0^n(1) = \ast$ and $d_0^n(k) = k-1$ for $1 < k \leq n$. 

For $1 \leq i \leq n-1$, let $\theta_i^n: \fin{n} \to \fin{n}$ be given by $P(\theta_i^n) = r_i^{n-1}$. This is the map that swaps $i$ and $i+1$:
\[ \theta_i^n(k) = \begin{cases}
    k & \mbox{if } 1 \leq k < i,\\
    i+1 & \mbox{if } k = i,\\
    i & \mbox{if } k = i+1,\\
    k & \mbox{if } i+1 < k \leq n.
\end{cases}
\]

\begin{lem}
The maps $s_i^n$, $d_i^n$, and $\theta_i^n$ generate $\Phi_\ast$.
\end{lem}
\begin{proof}
Since $P$ is faithful, it suffices to show that the image of $P$ is generated by the coface maps $\delta_i^n$ for $n \geq 1$ and $1 \leq i \leq n$, all of the coboundary maps $\sigma_i^n$, and the main transpositions $r_i^n$ for $1 \leq i \leq n$.

Let $g: \underline{n} \to \underline{m}$ be a map such that $g(0) = 0$. Let $g = h\rho$ be a factorization of $g$ into a permutation $\rho: \underline{n} \to \underline{n}$ and a monotonic map $h: \underline{n} \to \underline{m}$. Let $a = \rho(0)$ and $b = \rho^{-1}(0)$. Using the monotonicity of $h$, we have $g(b) = h(0)  \leq h(a) = g(0) = 0$, so $g(b) = 0$. As a result, if we define $\rho': \underline{n} \to \underline{n}$ by $\rho'(0)=0$, $\rho'(b)=a$, and $\rho'(i) = \rho(i)$ for other values of $i$, then we obtain a new factorization $g = h \rho'$.

Since $\rho'$ is a permutation such that $\rho'(0)=0$, it can be written as a composition of $r_i^n$ for $1 \leq i \leq n$. Since $h$ is a monotonic map such that $h(0)=0$, it can be written as a composition of $\delta_i^k$ for $1 \leq i \leq k$ and $\sigma_i^n$.
\end{proof}

Since $P$ is faithful, the relations satisfied by generators of $\Phi_\ast$ are exactly the relations satisfied by their images in $\Phi$. The relations induced from the cosimplicial relations \eqref{eqn:cosimp1}--\eqref{eqn:cosimp3} are as follows:
\begin{align}\label{eqn:simp1}
    s_i s_j &= s_{j+1} s_i, \;\;\;\;\;\; i \leq j,\\ \label{eqn:simp2}
    d_i d_j &= d_{j-1} d_{i}, \;\;\;\;\;\; i < j, \\ 
    d_i s_j &= \begin{cases}
        s_{j-1} d_i, & i < j, \\
        \id, & i = j, j+1, \\
        s_j d_{i-1}, & i > j+1.
    \end{cases} \label{eqn:simp3}
\end{align}
We note that these are the \emph{simplicial} relations, except that the final face maps $d_n^n$ are not included.

The relations induced from \eqref{eqn:moore1}--\eqref{eqn:moore3} are as follows:
\begin{align} \label{eqn:moore1b}
    (\theta_i)^2 &= \id, \\ \label{eqn:moore2b}
    \theta_i \theta_j \theta_i &= \theta_j \theta_i \theta_j, \;\;\; i = j-1, \\
    \theta_i \theta_j &= \theta_j \theta_i, \;\;\;\;\;\; i < j-1. \label{eqn:moore3b}
\end{align}
These are again the Moore relations.

The relations induced from \eqref{eqn:mixed1}--\eqref{eqn:mixed3} are as follows:
\begin{align}\label{eqn:mixed1b}
    \theta_i s_j &= \begin{cases}
        s_j \theta_i, & i < j, \\
        s_{i-1}, & i = j, \\
        s_{i}, & i = j+1 \\        
        s_j \theta_{i-1}, & i>j+1,
    \end{cases} \\ \label{eqn:mixed2b}
    \theta_i d_j &= \begin{cases}
        d_j \theta_i, & i < j-1,\\
        d_i \theta_{i+1} \theta_i, & i = j-1,\\
        d_{i+1}\theta_i \theta_{i+1}, & i=j,\\
        d_j \theta_{i+1}, & i>j,
    \end{cases} \\
    d_i \theta_i &= d_i.\label{eqn:mixed3b}
\end{align}

\begin{rmk}
    As noted above, \eqref{eqn:simp1}--\eqref{eqn:simp3} look like the simplicial relations for face maps $d_i^n$ and degeneracy maps $s_i^n$, except that the generators we have given for $\Phi_\ast$ do not include the final face maps $d_n^n$. This situation is remedied by defining
    \begin{equation} \label{eqn:dn}
    d_n^n = d_0^n \theta_1^n \cdots \theta_{n-1}^n.
    \end{equation}
    Then one can use \eqref{eqn:simp1}--\eqref{eqn:mixed3b} to show that \eqref{eqn:simp2}, \eqref{eqn:simp3}, and \eqref{eqn:mixed2b} hold for $d_n^n$ as well.
\end{rmk}

\subsection{\texorpdfstring{$\Gamma$-sets}{Gamma-sets}} \label{sec:gamma}
Recall that a $\Gamma$-set is defined to be a functor $\Phi_\ast \to \Set$. The following theorem immediately follows from the results of Section \ref{subsec:gen_rel_Phistar}.
\begin{thm}\label{thm:finstar}
A $\Gamma$-set is equivalent to a simplicial set $X_\bullet$, equipped with an action of the symmetric group $S_n$ on $X_n$ for each $n$, such that the relations \eqref{eqn:mixed1b}--\eqref{eqn:dn} are satisfied.
\end{thm}

\begin{rmk}\label{rmk:cutcomp}
The fact that any $\Gamma$-set has an underlying simplicial set can be understood conceptually by the fact that a functor $\Phi_\ast \to \Set$ can be pulled back along the functor $\Cut: \Delta^{\op} \to \Phi_\ast$. One can see that this pullback agrees with the simplicial structure described in Section \ref{subsec:gen_rel_Phistar} by directly applying $\Cut$ to the coface and codegeneracy maps in $\Delta$, and seeing that their images are exactly the face and degeneracy maps in $\Phi_\ast$.
\end{rmk}

\begin{rmk}
We warn the reader that there are structures on simplicial sets that superficially seem similar to $\Gamma$-sets but are in fact different. A \emph{symmetric simplicial set} is defined (see, e.g., \cite{Grandis}) as a functor $X:\Phi^\op\to \Set$. Similarly, there are \emph{$\Delta\mathfrak{S}$-sets} where $\Delta\mathfrak{S}$ is the \emph{symmetric crossed simplicial group} (see, e.g., \cite[Theorem 6.1.4]{Loday}, \cite[Example 6]{fiedorowicz-loday}). In both of these situations, one has a simplicial set $X_\bullet$ where $X_n$ carries an action of the symmetric group $S_{n+1}$.
	
As a point of contrast between $\Gamma$-sets and these other notions, one can see (either by directly applying $\Cut$ or by deducing from the relations in Section \ref{subsec:gen_rel_Phistar}) that $d_0^1 = d_1^1$ in $\Phi_\ast$, whereas the above two structures admit examples that do not satisfy this equation. For example, the nerve of a groupoid has the structure of a symmetric simplicial set \cite{Grandis:Higher} but does not satisfy $d_0^1 = d_1^1$ in general. 
\end{rmk}

\subsection{Commutative pseudomonoids in \texorpdfstring{$\Span$}{Span}}

This subsection is devoted to the proof of the following result.

\begin{thm}\label{thm:equiv_phistar_comm}
    Let $X_\bullet$ be a $2$-Segal set. Then there is a one-to-one correspondence between $\Gamma$-structures on $X_\bullet$ and equivalence classes of commutative structures on the corresponding pseudomonoid in $\Span$.
\end{thm}

\subsubsection*{A \texorpdfstring{$\Gamma$}{Gamma}-structure determines a commutative structure}

Suppose that $X_\bullet$ is a $2$-Segal $\Gamma$-set. For simplicity of notation, we will write $\theta = \theta_1^2$. From \eqref{eqn:moore1b}, \eqref{eqn:mixed3b}, and \eqref{eqn:dn}, we have that $d_0^2 \theta = d_2^2$, $d_1^2 \theta = d_1^2$, and $d_2^2 \theta = d_0^2$. Thus $\theta: X_2 \to X_2$ defines a map of spans (i.e.\ a $2$-morphism in $\Span$) from the multiplication morphism $\mu$ (see \eqref{eqn:mu}) to $\mu \circ \rho_{X,X}$, where the latter is canonically identified with the span
\begin{equation}\label{eqn:murho}
\begin{tikzcd}
	X_1 \times X_1 &[2em] \arrow[l,"{(d_0,d_2)}"'] X_2 \arrow[r,"d_1"] & X_1.
\end{tikzcd}
\end{equation}
We will take $\theta$ to play the role of $\gamma$ in the definition of a commutative pseudomonoid. The symmetry condition in Section \ref{subsec:comm_pseudomonoids} follows from the fact that $\theta^2 = \id$.

Recall from Section \ref{sec:2Segal} that we can use the $2$-Segal conditions to identify $X_3$ with either $X_2 \bitimes{d_1}{d_2} X_2$ or $X_2 \bitimes{d_0}{d_1} X_2$, corresponding to the two triangulations of the square in Figure \ref{fig:square}. These identifications allow us to visualize the maps $\theta_1^3$ and $\theta_2^3$ as in Figure \ref{fig:theta1theta2}.
\begin{figure}[th]
\begin{center}
\begin{tikzpicture}[scale=0.8]
    \begin{scope}
        \draw (-135:1) node[vertex] (0) {} node[anchor=north east] {0}
        -- (-45:1) node[vertex] (3) {} node[anchor=north west] {3} 
        -- (45:1) node[vertex] (2) {} node[anchor=south west] {2}
        -- (135:1) node[vertex] (1) {} node[anchor=south east] {1};
        \draw (0) -- (1);
        \draw (0) -- (2);
        \node at (135:.4) {$\psi_3$};
        \node at (-45:.4) {$\psi_1$};
    \end{scope}
    \node at (2,0) {$\xmapsto[\phantom{xxxxxxx}]{\theta_1^3}$};
    \begin{scope}[shift={(4,0)}]
        \draw (135:1) node[vertex] (1) {} node[anchor=south east] {1}
        -- (-135:1) node[vertex] (0) {} node[anchor=north east] {0}
        -- (-45:1) node[vertex] (3) {} node[anchor=north west] {3}
        -- (45:1) node[vertex] (2) {} node[anchor=south west] {2};
        \draw (1) -- (2);
        \draw (0) -- (2);
        \node at (135:.4) {${\scriptstyle \theta \psi_3}$};
        \node at (-45:.4) {$\psi_1$};
    \end{scope}   
    \begin{scope}[shift={(9,0)}]
        \draw (135:1) node[vertex] (1) {} node[anchor=south east] {1}
        -- (-135:1) node[vertex] (0) {} node[anchor=north east] {0}
        -- (-45:1) node[vertex] (3) {} node[anchor=north west] {3}
        -- (45:1) node[vertex] (2) {} node[anchor=south west] {2};
        \draw (1) -- (2);
        \draw (1) -- (3);
        \node at (45:.4) {$\psi_0$};
        \node at (-135:.4) {$\psi_2$};
    \end{scope}    
    \node at (11,0) {$\xmapsto[\phantom{xxxxxxx}]{\theta_2^3}$};
    \begin{scope}[shift={(13,0)}]
        \draw (45:1) node[vertex] (2) {} node[anchor=south west] {2}
        -- (135:1) node[vertex] (1) {} node[anchor=south east] {1}
        -- (-135:1) node[vertex] (0) {} node[anchor=north east] {0}
        -- (-45:1) node[vertex] (3) {} node[anchor=north west] {3};
        \draw (2) -- (3);
        \draw (1) -- (3);
        \node at (45:.4) {$\scriptstyle \theta \psi_0$};
        \node at (-135:.4) {$\psi_2$};
    \end{scope}        
\end{tikzpicture}
\end{center}
    \caption{The maps $\theta_1^3$ and $\theta_2^3$ are uniquely determined by the equations $d_1^3 \theta_1^3 = d_1^3$, $d_0^3 \theta_2^3 = \theta d_0^3$,
$d_3^3 \theta_1^3 = \theta d_3^3$, and $d_2^3 \theta_2^3 = d_2^3$.}
    \label{fig:theta1theta2}
\end{figure}
There is another map $c: X_3 \to X_3$ that will be useful to define. It is given by $d_1^3 c = \theta d_2^3$ and $d_3 c = d_0$; see Figure \ref{fig:c}.
\begin{figure}[th]
\begin{center}
\begin{tikzpicture}[scale=0.8]
    \begin{scope}
        \draw (135:1) node[vertex] (1) {} node[anchor=south east] {1}
        -- (-135:1) node[vertex] (0) {} node[anchor=north east] {0}
        -- (-45:1) node[vertex] (3) {} node[anchor=north west] {3}
        -- (45:1) node[vertex] (2) {} node[anchor=south west] {2};
        \draw (1) -- (2);
        \draw (1) -- (3);
        \node at (45:.4) {$\psi_0$};
        \node at (-135:.4) {$\psi_2$};
    \end{scope}  
    \node at (2,0) {$\xmapsto[\phantom{xxxxxxx}]{c}$};
    \begin{scope}[shift={(4,0)}]
        \draw (-135:1) node[vertex] (0) {} node[anchor=north east] {0}
        -- (-45:1) node[vertex] (3) {} node[anchor=north west] {3} 
        -- (45:1) node[vertex] (2) {} node[anchor=south west] {2}
        -- (135:1) node[vertex] (1) {} node[anchor=south east] {1};
        \draw (0) -- (1);
        \draw (0) -- (2);
        \node at (135:.4) {$\psi_0$};
        \node at (-45:.4) {${\scriptstyle \theta\psi_2}$};
    \end{scope} 
\end{tikzpicture}
\end{center}
    \caption{The map $c: X_3 \to X_3$, given by $d_1^3 c = \theta d_2^3$ and $d_3 c = d_0$.}
    \label{fig:c}
\end{figure}

We now turn to the {hexagon equation} (\ref{eqn:hexagon}). We first observe that canonical identifications can be made such that the $2$-morphisms $\rho_{1,\mu}$ and $R_{X|XX}$ correspond to identity maps of spans. This allows us to collapse the corresponding maps in the hexagon equation, so that it actually has the form of a hexagon.

Next, we observe that every composite morphism in the hexagon equation can be canonically identified with a span of the form
\begin{equation}
\begin{tikzcd}
	X_1 \times X_1 \times X_1 &[1em] \arrow[l] X_3 \arrow[r,"e_{\out}"] & X_1,
\end{tikzcd}
\end{equation}
where the map on the left is some permutation of $(e_1,e_2,e_3)$, and where there is a chosen identification of $X_3$ with either $X_2 \bitimes{d_1}{d_2} X_2$ or $X_2 \bitimes{d_0}{d_1} X_2$. With this understanding, we can realize the hexagon equation as requiring the commutativity of the diagram of triangulated squares in Figure \ref{fig:hexagon}.
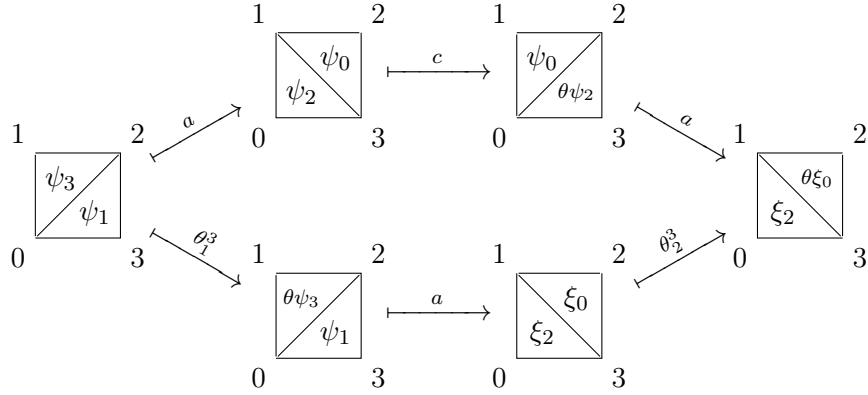
\begin{figure}[th]
\begin{center}
\begin{tikzpicture}[scale=0.8]
    \begin{scope}
        \draw (-135:1) node[vertex] (0) {} node[anchor=north east] {0}
        -- (-45:1) node[vertex] (3) {} node[anchor=north west] {3} 
        -- (45:1) node[vertex] (2) {} node[anchor=south west] {2}
        -- (135:1) node[vertex] (1) {} node[anchor=south east] {1};
        \draw (0) -- (1);
        \draw (0) -- (2);
        \node at (135:.4) {$\psi_3$};
        \node at (-45:.4) {$\psi_1$};
    \end{scope}
    \node[rotate=30] at (2,1) {$\xmapsto[\phantom{xxxxxxx}]{a}$};
    \begin{scope}[shift={(4,2)}]
        \draw (135:1) node[vertex] (1) {} node[anchor=south east] {1}
        -- (-135:1) node[vertex] (0) {} node[anchor=north east] {0}
        -- (-45:1) node[vertex] (3) {} node[anchor=north west] {3}
        -- (45:1) node[vertex] (2) {} node[anchor=south west] {2};
        \draw (1) -- (2);
        \draw (1) -- (3);
        \node at (45:.4) {$\psi_0$};
        \node at (-135:.4) {$\psi_2$};
    \end{scope}  
    \node at (6,2) {$\xmapsto[\phantom{xxxxxxx}]{c}$};
    \begin{scope}[shift={(8,2)}]
        \draw (-135:1) node[vertex] (0) {} node[anchor=north east] {0}
        -- (-45:1) node[vertex] (3) {} node[anchor=north west] {3} 
        -- (45:1) node[vertex] (2) {} node[anchor=south west] {2}
        -- (135:1) node[vertex] (1) {} node[anchor=south east] {1};
        \draw (0) -- (1);
        \draw (0) -- (2);
        \node at (135:.4) {$\psi_0$};
        \node at (-45:.4) {${\scriptstyle \theta\psi_2}$};
    \end{scope}
    \node[rotate=-30] at (10,1) {$\xmapsto[\phantom{xxxxxxx}]{a}$};
    \begin{scope}[shift={(12,0)}]
        \draw (45:1) node[vertex] (2) {} node[anchor=south west] {2}
        -- (135:1) node[vertex] (1) {} node[anchor=south east] {1}
        -- (-135:1) node[vertex] (0) {} node[anchor=north east] {0}
        -- (-45:1) node[vertex] (3) {} node[anchor=north west] {3};
        \draw (2) -- (3);
        \draw (1) -- (3);
        \node at (45:.4) {${\scriptstyle \theta \xi_0}$};
        \node at (-135:.4) {$\xi_2$};
    \end{scope} 
    \node[rotate=-30] at (2,-1) {$\xmapsto[\phantom{xxxxxxx}]{\theta_1^3}$};
    \begin{scope}[shift={(4,-2)}]
        \draw (-135:1) node[vertex] (0) {} node[anchor=north east] {0}
        -- (-45:1) node[vertex] (3) {} node[anchor=north west] {3} 
        -- (45:1) node[vertex] (2) {} node[anchor=south west] {2}
        -- (135:1) node[vertex] (1) {} node[anchor=south east] {1};
        \draw (0) -- (1);
        \draw (0) -- (2);
        \node at (135:.4) {${\scriptstyle \theta \psi_3}$};
        \node at (-45:.4) {$\psi_1$};
    \end{scope}
    \node at (6,-2) {$\xmapsto[\phantom{xxxxxxx}]{a}$};
    \begin{scope}[shift={(8,-2)}]
        \draw (135:1) node[vertex] (1) {} node[anchor=south east] {1}
        -- (-135:1) node[vertex] (0) {} node[anchor=north east] {0}
        -- (-45:1) node[vertex] (3) {} node[anchor=north west] {3}
        -- (45:1) node[vertex] (2) {} node[anchor=south west] {2};
        \draw (1) -- (2);
        \draw (1) -- (3);
        \node at (45:.4) {$\xi_0$};
        \node at (-135:.4) {$\xi_2$};
    \end{scope}
    \node[rotate=30] at (10,-1) {$\xmapsto[\phantom{xxxxxxx}]{\theta_2^3}$};
\end{tikzpicture}
\end{center}
    \caption{The hexagon equation.}
    \label{fig:hexagon}
\end{figure}
If we identify each node of Figure \ref{fig:hexagon} with $X_3$, then each associator becomes the identity map, and the hexagon equation becomes $c = \theta_2^3 \theta_1^3$. By the definition of $c$, this is equivalent to the pair of equations
\begin{align}\label{eqn:hexagonequivalent}
    d_1^3 \theta_2^3 \theta_1^3 &= \theta_1^2 d_2^3, & d_3^3 \theta_2^3 \theta_1^3 &= d_0^3,
\end{align}
which hold by \eqref{eqn:mixed2b} and \eqref{eqn:dn}.

\subsubsection*{A commutative structure determines a \texorpdfstring{$\Gamma$}{Gamma}-structure}

Suppose that $X_\bullet$ is a $2$-Segal set for which the corresponding pseudomonoid in $\Span$ is equipped with a commutative structure, i.e.\ a map of spans $\gamma: \mu \Rightarrow \mu \circ \rho_{X,X}$ satisfying the {symmetry} (\ref{eqn:symmetry}) and {hexagon} (\ref{eqn:hexagon}) equations. In low degrees, the construction of a $\Gamma$-structure is simply the reverse of the other direction.

We define $\theta = \theta_1^2: X_2 \to X_2$ to be the map obtained from $\gamma$ upon identifying $\mu \circ \rho_{X,X}$ with the span \eqref{eqn:murho}.
\begin{lem}
    \begin{enumerate}
        \item $(\theta_1^2)^2 = \id$.
        \item The following face map compatibility conditions hold:
            \begin{itemize}
                \item $d_0^2 \theta_1^2 = d_2^2$,
                \item $d_1^2 \theta_1^2 = d_1^2$,
                \item $d_2^2 \theta_1^2 = d_0^2$.
            \end{itemize}
        \item The following degeneracy map compatibility conditions hold:
            \begin{itemize}
                \item $\theta_1^2 s_0^1 = s_1^1$,
                \item $\theta_1^2 s_1^1 = s_0^1$.
            \end{itemize}
    \end{enumerate}
\end{lem}
\begin{proof}
    The first equation holds as a result of symmetry of $\gamma$. The face map compatibility conditions hold as a result of the fact that $\gamma$ is a map of spans. The degeneracy map compatibility conditions follow from the unitality conditions \eqref{eqn:unitality} and the face map compatibility conditions.
\end{proof}
Next, we define $\theta_1^3, \theta_2^3: X_3 \to X_3$ as in Figure \ref{fig:theta1theta2}, and we define $c: X_3 \to X_3$ as in Figure \ref{fig:c}. Then the hexagon equation has the form in Figure \ref{fig:hexagon} and is equivalent to the pair of equations \eqref{eqn:hexagonequivalent}.
\begin{lem}\label{lem:theta3}
    The maps $\theta_1^3$ and $\theta_2^3$ satisfy all applicable cases of relations \eqref{eqn:moore1b}--\eqref{eqn:dn}.
\end{lem}
\begin{proof}
The equations $(\theta_1^3)^2 = (\theta_2^3)^2 = \id$ hold by construction. All of the face map compatibility conditions hold either by definition of $\theta_1^3$, $\theta_2^3$, or by the equations \eqref{eqn:hexagonequivalent} that correspond to the hexagon equation.

To check that $\theta_1^3 \theta_2^3 \theta_1^3 = \theta_2^3 \theta_1^3 \theta_2^3$, we use the fact that, by the $2$-Segal conditions, it is sufficient to show that both sides have the same images under $d_1^3$ and $d_3^3$. This is indeed the case, since by the face map compatibility conditions, we have
\begin{align*}
    d_1^3 \theta_2^3 \theta_1^3 \theta_2^3 &= \theta_1^2 d_2^3 =     d_1^3 \theta_1^3 \theta_2^3 \theta_1^3,\\
    d_3^3 \theta_2^3 \theta_1^3 \theta_2^3 &= \theta_1^2 d_0^3 = d_3^3 \theta_1^3 \theta_2^3 \theta_1^3.
\end{align*}
The degeneracy map compatibility conditions can be proved similarly, using relations that have already been established. For example, since $d_1^3 \theta_1^3 s_1^2 = d_1^3 s_1^2 = \id = d_1^3 s_0^2$ and $d_3^3 \theta_1^3 s_1^2 = \theta_1^2 d_3^3 s_1^2 = \theta_1^2 s_1^1 d_2^2 = s_0^1 d_2^2 = d_3^3 s_0^2$, we conclude that $\theta_1^3 s_1^2 = s_0^2$.
\end{proof}

More generally, we define $\theta_i^n$ for $n \geq 2$ as follows, using the graphical calculus of Section \ref{sec:faceanddegen}. Choose a triangulation $\mathcal{T}$ of the $(n+1)$-gon that includes the triangle with vertices $\{i-1, i, i+1\}$. For $\psi \in X_n$, consider the image $(\xi_1, \dots, \xi_{n-1}) \in X_2 \times_{X_1} \cdots \times_{X_1} X_2$ under the $2$-Segal map $\hat{\mathcal{T}}$ in \eqref{eqn:triangulation}. Apply $\theta = \theta_1^2$ to the component that corresponds to the triangle with vertices $\{i-1, i, i+1\}$. Then the corresponding element of $X_n$ is $\theta_i^n \psi$. This is illustrated in Figure \ref{fig:highertheta}.

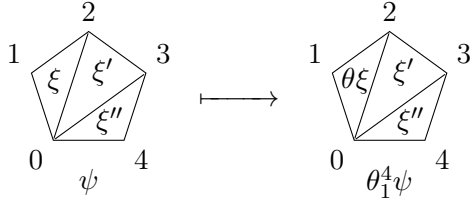
\begin{figure}[th]
\begin{center}
\begin{tikzpicture}[scale=0.8]
    \begin{scope}
        \draw (-126:1) node[vertex] (0) {} node[anchor=north east] {0}
        -- (-54: 1) node[vertex] (4) {} node[anchor=north west] {4}
        -- (18:1) node[vertex] (3) {} node[anchor=south west] {3}
        -- (90:1) node[vertex] (2) {} node[anchor=south] {2}
        -- (162:1) node[vertex] (1) {} node[anchor=south east] {1}
        -- cycle;
        \draw (0) -- (2);
        \draw (0) -- (3);
        \node at (162:.6) {$\xi$};
        \node at (54:.4) {$\xi'$};
        \node at (-54:.6) {$\xi''$};
        \node at (0,-1.5) {$\psi$};
    \end{scope}
    \node at (2.5,0) {$\xmapsto{\phantom{xxxxx}}$};
    \begin{scope}[shift={(5,0)}]
        \draw (-126:1) node[vertex] (0) {} node[anchor=north east] {0}
        -- (-54: 1) node[vertex] (4) {} node[anchor=north west] {4}
        -- (18:1) node[vertex] (3) {} node[anchor=south west] {3}
        -- (90:1) node[vertex] (2) {} node[anchor=south] {2}
        -- (162:1) node[vertex] (1) {} node[anchor=south east] {1}
        -- cycle;
        \draw (0) -- (2);
        \draw (0) -- (3);
        \node at (162:.6) {$\theta \xi$};
        \node at (54:.4) {$\xi'$};
        \node at (-54:.6) {$\xi''$};
        \node at (0,-1.5) {$\theta_1^4 \psi$};
    \end{scope}
\end{tikzpicture}
\end{center}
    \caption{Graphical calculus for $\theta_i^n$. To produce $\theta_1^4 \psi$, we choose a triangulation that includes the triangle with vertices $\{0,1,2\}$, and we apply $\theta$ to the corresponding $2$-simplex.}
    \label{fig:highertheta}
\end{figure}

\begin{lem}
    The maps $\theta_i^n$ satisfy relations \eqref{eqn:moore1b}--\eqref{eqn:dn}.
\end{lem}
\begin{proof}
    Relation \eqref{eqn:moore1b} follows from the fact that $\theta^2 = \id$. For \eqref{eqn:moore3b}, we observe that, when $i < j-1$, one can choose a triangulation that simultaneously includes the triangles $\{i-1,i,i+1\}$ and $\{j-1,j,j+1\}$. Then it is clear from the graphical calculus that, in this case, $\theta_i$ and $\theta_j$ commute.

    When $i=j-1$, we have that the maps $\theta_i$ and $\theta_j = \theta_{i+1}$ only affect the quadrilateral with vertices $\{i-1,i,i+i,i+2\}$. Thus \eqref{eqn:moore2b} can be deduced from the case $n=3$, which holds by Lemma \ref{lem:theta3}.

    Relations \eqref{eqn:mixed1b}--\eqref{eqn:mixed3b} can be similarly proven with the graphical calculus. In every case, one either has operations that involve non-overlapping triangles and thus commute (possibly with an index shift) or operations that take place within a quadrilateral and thus follow from the $n=3$ case. For the face map compatibility conditions, we have that $d_j$ deletes the triangle with vertices $\{j-1,j,j+1\}$, and $\theta_i$ applies $\theta$ to the triangle with vertices $\{i-1,i,i+1\}$. If $i < j-1$, then it is possible to choose a triangulation that includes both triangles, and one can see that the operations commute. If $i = j-1$, then the two operations take place within the quadrilateral with vertices $\{i-1,i,i+1,i+2\}$, so it follows from \eqref{eqn:hexagonequivalent} that $\theta_i d_j = d_i \theta_{i+1} \theta_i$. The arguments for the remaining cases, as well as the degeneracy map compatibility conditions, are similar.

    Finally, we prove \eqref{eqn:dn} by induction, as follows. The base case $n=3$ holds by the hexagon equation (see \eqref{eqn:hexagonequivalent}). We then use the simplicial relations, \eqref{eqn:mixed2b}--\eqref{eqn:mixed3b}, and the inductive hypothesis to see that
    \begin{equation*}
        \begin{split}
            d_0^{n-1} d_0^n \theta_1^n \cdots \theta_{n-1}^n &= d_0^{n-1} d_1^n \theta_1^n \cdots \theta_{n-1}^n \\
            &= d_0^{n-1} d_1^n \theta_2^n \cdots \theta_{n-1}^n \\
            &= d_0^{n-1} d_0^n \theta_2^n \cdots \theta_{n-1}^n \\
            &= d_0^{n-1} \theta_1^{n-1} \cdots \theta_{n-2}^{n-1} d_0^n \\    
            &= d_{n-1}^{n-1} d_0^n \\
            &= d_0^{n-1} d_n^n
        \end{split}
    \end{equation*}
    and 
    \begin{equation*}
        \begin{split}
            d_{n-2}^{n-1} d_0^n \theta_1^n \cdots \theta_{n-1}^n &= d_0^{n-1} d_{n-1}^n \theta_1^n \cdots \theta_{n-1}^n \\
            &= d_0^{n-1} \theta_1^{n-1} \cdots \theta_{n-3}^{n-1} d_{n-1}^n \theta_{n-2}^n \theta_{n-1}^n \\
            &= d_0^{n-1} \theta_1^{n-1} \cdots \theta_{n-3}^{n-1} \theta_{n-2}^{n-1} d_{n-2}^n \\
            &= d_{n-1}^{n-1} d_{n-2}^n \\
            &= d_{n-2}^{n-1} d_n^n.
        \end{split}
    \end{equation*}
By the $2$-Segal property, \eqref{eqn:dn} follows.
\end{proof}

\subsection{Examples} \label{subsec:sym_egs}

\begin{example}[Commutative partial monoids]
Recall from Example \ref{ex:partialmonoids} that the nerve of a partial monoid is a $2$-Segal set. A partial monoid $M$ is \emph{commutative} if $x \cdot x' = x' \cdot x$ for all $x,x' \in M$. 

If $M$ is a commutative partial monoid, then each $N_n M$ admits an $S_n$-action given by permutation of the components. It is then straightforward to check that the compatibility conditions \eqref{eqn:mixed1b}--\eqref{eqn:dn} hold. Thus $N_\bullet M$ has the structure of a $2$-Segal $\Gamma$-set.
\end{example}

\begin{example}[Graph partitions]
Recall (see Example \ref{ex:graphpartitions}) the construction from \cite[Example 2.3]{BOORS} of a $2$-Segal set $X(G)_\bullet$ associated to a graph $G$. The elements of $X(G)_n$ are of the form $(H; V_1, \dots, V_n)$, where $H$ is a subgraph of $G$ and $(V_1, \dots, V_n)$ is a partition of the set of vertices of $H$. In particular, in low degrees, $X(G)_0 = \pt$ and $X(G)_1$ consists of subgraphs $H \subseteq G$. Each $X(G)_n$ admits an $S_n$-action given by permutation of the $V_i$, giving $X(G)_\bullet$ the structure of a $\Gamma$-set.

In fact, we feel that in this example it is more straightforward to directly define the $\Gamma$-structure as a functor $\Phi_\ast \to \Set$ than it is to define the simplicial structure together with symmetric group actions satisfying the compatibility conditions. Namely, given a morphism $f:\langle n\rangle \to \langle m\rangle $  in $\Phi_\ast$, we define a map 
	\[
	\begin{tikzcd}
		f_\ast:&[-3em] X(G)_n \arrow[r] & X(G)_m 
	\end{tikzcd}
	\]
	as follows. 
	\[
	f_\ast(H;V_1,\ldots,V_n)=(H^\prime; V_1',\ldots,V_m')
	\]
	where 
	\[
	V_j'=\bigcup_{i\in f^{-1}(j)} V_i 
	\]
	and $H^\prime$ is the full subgraph of $H$ on the vertices $\bigcup_{j=1}^m V_j'$. One can check that this assignment is functorial, and so defines a $\Gamma$-set. 
	
	One can also see that this construction coincides with that of \cite{BOORS} by giving its values on the face and degeneracy maps from \ref{subsec:gen_rel_Phistar}. Explicitly, 
	\begin{itemize}
		\item For $s_i:\langle n\rangle \to \langle n+1\rangle$, we have 
		\[
		(s_i)_\ast(H;V_1,\ldots, V_n)= (H; V_1,\ldots,V_{i-1},\varnothing, V_i,\ldots, V_n)
		\]
		\item For $0<i<n$ and $d_i:\langle n\rangle \to \langle n-1\rangle$, we have 
		\[
		(d_i)_\ast(H;V_1,\ldots, V_n) = (H;V_1,\ldots, V_{i-1}, V_i\cup V_{i+1}, V_{i+2},\ldots, V_n)
		\]
		\item For $d_0: \langle n\rangle \to \langle n-1\rangle$ we have 
		\[
		(d_0)_\ast(H;V_1,\ldots,V_n)= (H^\prime; V_2,\ldots,V_n) 
		\]
		where $H^\prime$ is the full subgraph of $H$ on $V(H)\setminus V_1$.
		\item For $d_n: \langle n\rangle \to \langle n-1\rangle$ we have 
		\[
		(d_n)_\ast(H;V_1,\ldots,V_n)= (H^\prime; V_1,\ldots,V_{n-1}) 
		\]
		where $H^\prime$ is the full subgraph of $H$ on $V(H)\setminus V_n$. 
	\end{itemize}
	As such, the underlying simplicial set of $X(G)$ is precisely the simplicial set of \cite[Example 2.3]{BOORS} and so is 2-Segal.
\end{example}

\begin{example}[Graph complexes]
	There is a variant of the previous example, more in line with the original construction of \cite[Example 1.1.5]{GKT_comb}. Instead of fixing a graph $G$ whose subgraphs we consider, we can instead consider a simplicial object
	\[
	\begin{tikzcd}
		X: &[-3em] \Phi_\ast \arrow[r] & \Set
	\end{tikzcd}
	\]
	for which $X_n$ is the set\footnote{This definition clearly runs into some substantial size issues, which can be addressed by appealing to alternate foundations, e.g., Grothendieck Universes.} of equivalence classes of graphs  equipped with an $n$-component partition of their vertex set. We can define $X$ on morphisms in a manner identical to the preceding example, and so obtain a 2-Segal $\Gamma$-set. 
	
	The original example in \cite{GKT_comb} might lead us to suppose that this 2-Segal $\Gamma$-set is a categorification of Schmitt's coalgebra\footnote{Notice that the obvious duality on $\Span$ obtained by reading spans backwards canonically identifies (pseudo) algebras and (pseudo) coalgebras in $\Span$. Thus both correspond to 2-Segal simplicial objects.} of graphs from \cite{Schmitt}, and that the fact that we can promote this 2-Segal object to a $\Gamma$-set is a result of the cocommutativity of this coalgebra, as remarked in \cite{Schmitt}. 
	
	However, this is too coarse, since there are abstractly isomorphic partitioned graphs which represent different terms in Schmitt's comultiplication (just as remarked about the Butcher–Connes–Kreimer bialgebra in \cite[2.5.4]{GKT_comb}). However, we can rectify this by defining a 2-Segal \emph{$\Gamma$-groupoid}
	\[
	\begin{tikzcd}
		\scr{X}: &[-3em] \Phi_\ast \arrow[r] & \on{Grpd} 
	\end{tikzcd}
	\] 
	In which $\scr{X}_n$ is the \emph{groupoid} of graphs equipped with a $n$-component partition of their vertex set. While this $\Gamma$-groupoid goes beyond the scope of this paper, it adds another layer of evidence that our results generalize naturally to the higher-categorical setting. 
\end{example}

	\printbibliography
	
\end{document}